\documentclass[12pt]{amsart}
\usepackage{amscd, amsmath, amssymb}
\usepackage{fullpage}
\usepackage[misc,geometry]{ifsym}
\usepackage{paralist}

\theoremstyle{plain}
\newtheorem{thm}{Theorem}[section]
\newtheorem{lem}[thm]{Lemma}
\newtheorem{pro}[thm]{Proposition}
\theoremstyle{definition}
\newtheorem{defn}[thm]{Definition}
\newtheorem{ex}[thm]{Example}
\newtheorem{rem}[thm]{Remark}

\numberwithin{equation}{section}

\newcommand{\R}{\mathbb{R}}
\newcommand{\N}{\mathbb{N}}

\newcommand{\de}{\partial}
\newcommand{\LL}{\mathcal{L}}
\renewcommand{\AA}{\mathcal{A}}

\newcommand{\OO}{\mathcal{O}}
\newcommand{\clOO}{\overline{\mathcal{O}}}
\newcommand{\bldu}{\mathbf{u}}
\newcommand{\bldw}{\mathbf{w}}
\newcommand{\bldz}{\mathbf{z}}
\newcommand{\GG}{\mathcal{G}}
\renewcommand{\d}{\mathrm{d}}

\usepackage{color}

\begin{document}

\title[Positive solutions for nonlocal elliptic systems]{Nonzero positive solutions of elliptic systems with gradient dependence and functional BCs}  

\date{\today}

\author[S. Biagi]{Stefano Biagi}
\address{Stefano Biagi, Dipartimento di Dipartimento di Ingegneria Industriale e 
Scienze Ma\-te\-ma\-ti\-che,
Universit\`a Politecnica delle Marche, Via Brecce Bianche, 60131 Ancona, Italy}%
\email{biagi@dipmat.univpm.it}%

\author[A. Calamai]{Alessandro Calamai}
\address{Alessandro Calamai, Dipartimento di Ingegneria Civile, Edile e Architettura,
Universit\`a Politecnica delle Marche, Via Brecce Bianche, 60131 Ancona, Italy}%
\email{calamai@dipmat.univpm.it}%

\author[G. Infante]{Gennaro Infante}
\address{Gennaro Infante, Dipartimento di Matematica e Informatica, Universit\`{a} della
Calabria, 87036 Arcavacata di Rende, Cosenza, Italy}%
\email{gennaro.infante@unical.it}%

\begin{abstract} 
We discuss, by topological methods, the solvability of systems of second-order 
elliptic differential equations subject to functional boundary conditions under the presence of gradient terms in the nonlinearities.
We prove the existence of non-negative solutions and provide a non-existence result. We present some examples to illustrate the applicability of the existence and non-existence results.
\end{abstract}



\maketitle

\section{Introduction}

In this paper we study the solvability of a system of second-order 
elliptic differential equations subject to functional boundary conditions (BCs for short).
Namely, we investigate parametric systems of the type
\begin{equation}
  \label{nellbvp-introduction}
 \left\{
\begin{array}{lll}
 \LL_k u_k=\lambda_k\,f_k(x,u_1,\ldots,u_m, \nabla u_1,\ldots,\nabla u_m)
 &  \text{in $\OO$} & \qquad (k=1,2,\ldots,m), \\[0.15cm]
 u_k(x)=\eta_k\,\zeta_k(x)\,h_{k}[u_1,\ldots,u_m] & \text{for $x\in\de\OO$}
 & \qquad (k=1,2,\ldots,m),
\end{array}
\right.
\end{equation}
where $m\geq 1$ is a fixed natural number, 
$\OO\subseteq\R^n$ is a bounded and connected open set of class $C^{1,\alpha}$ for some
 $\alpha\in(0,1)$, and $ \lambda_k,\,\eta_k$, $k = 1,\ldots, m$, are non-negative real parameters. Moreover
$\LL_1,\ldots,\LL_m$ are uniformly elliptic, second-order linear partial differential operators
(PDOs) in divergence form on $\OO$. 
That is, for $k = 1,\ldots, m$,
  \begin{equation*}
   \begin{split}
  \LL_k u & := 
   -\sum_{i,j = 1}^n
  \de_{x_i}\Big(a^{(k)}_{i,j}(x)\de_{x_j}u + b^{(k)}_i(x)u\Big)+
  \sum_{i = 1}^n c^{(k)}_i(x)\de_{x_i}u + d^{(k)}(x)u
  \end{split}
  \end{equation*}
where
\begin{itemize}
  \item
  the coefficient functions of $\LL_k$ belong to
  $C^{1,\alpha}(\clOO,\R)$;
  \item the matrix $A^{(k)}(x) := \big(a_{i,j}^{(k)}(x)\big)_{i,j}$
  is symmetric for every $x\in\OO$;
  \item $\LL_k$ is uniformly elliptic
  in $\OO$, i.e., there exists $\Lambda_k > 0$ such that
  $$\frac{1}{\Lambda_k}\|\xi\|^2\leq \sum_{i,j = 1}^na^{(k)}_{i,j}(x)\xi_i\xi_j\leq 
  \Lambda_k\|\xi\|^2\quad \text{for any $x\in\OO$ and $\xi\in\R^n\setminus\{0\}$}$$
  where $\|\xi\|$ stands for the Euclidean norm of $\xi\in\R^n$;
  \item for every non-negative function $\varphi\in C_0^\infty(\OO,\R)$ one has
  $$\int_{\mathcal{O}}
  \Big(d^{(k)}\varphi+{\textstyle\sum_{i = 1}^n}b^{(k)}_i\de_{x_i}\varphi\Big)\,\d x,\quad
  \int_{\mathcal{O}}
  \Big(d^{(k)}\varphi+{\textstyle\sum_{i = 1}^n}c_i^{(k)}\de_{x_i}\varphi\Big)\,\d x\geq 0.$$
 \end{itemize}

Furthermore, for every fixed $k = 1,\ldots, m$ we also assume that
\begin{itemize}
  \item $f_k$ is a real-valued continuous function defined on $\clOO\times\R^m\times\R^{nm}$;
  \item $h_k$ is a real-valued continuous functional defined on the space $C^1(\clOO,\R^m)$;
  \item $\zeta_k\in C^{1,\alpha}(\clOO,\R)$ and $\zeta_k\geq 0$ on $\OO$.
\end{itemize}

The system~\eqref{nellbvp-introduction} is quite general, and includes, for example, as a particular case
a Dirichlet boundary value problem for elliptic systems with gradient dependence of the form
  \begin{equation}~\label{laplace-intro}
 \left\{
\begin{array}{lll}
 -\Delta u_1=\lambda_1\,f_1(x,u_1,u_2, \nabla u_1,\nabla u_2) &  \text{in $\OO$} & \\
 -\Delta u_2=\lambda_2\,f_2(x,u_1,u_2, \nabla u_1,\nabla u_2) &  \text{in $\OO$} & \\
 u_1\big|_{\de \OO} = 0= u_2\big|_{\de \OO}
\end{array}
\right.
\end{equation}
Systems of nonlinear PDEs of this kind are widely studied in view of applications: in fact, the nonlinearities in \eqref{laplace-intro} may depend also on the gradient of the solution, and thus represent convection terms. These problems, in general, are not easily dealt with by means of variational methods.
Different approaches in the study of PDEs with gradient terms have been proposed: for example sub- and super-solutions, topological degree theory, mountain pass techniques.
We mention, for instance, the pioneering works of Amann and Crandall \cite{AmCra}, Br\'ezis and Turner \cite{BT},
Mawhin and Schmitt \cite{MawSchm, MawSchm-c}, Pokhozhaev \cite{Po}
and the more recent contributions \cite{Alves-Moussaoui, Radu-book,FGM, FY, GM, RuSu, WaDe, Yan}. See also the very recent survey \cite{Fig18} and references therein.

In this paper we adopt a topological approach, based on the classical notion of  fixed point index (see e.g.\ \cite{guolak}) for the existence result, Theorem \ref{thmsol} below, whereas we prove a non-existence result
via an elementary argument.
In some sense we follow a path established by Amman \cite{Amann-rev,AmCra} and successfully used by many authors in different contexts.
We point out that our approach applies not only to Dirichlet BCs but permits to consider (possibly nonlinear) functional  BCs, including the special cases of \emph{linear} (\emph{multi-point} or \emph{integral}) BCs of the form
  \begin{equation} \label{multipoint-intro}
h_{k}[u]=\sum_{j=1}^{m} \sum_{i=1}^{N} \left( \hat{\alpha}_{ijk}u_j(\omega_i)
+\sum_{l=1}^{n} \hat{\beta}_{ijkl}\de_{x_l} u_j(\tau_i) \right)
\end{equation}
or
\begin{equation} \label{integral-intro}
h_{k}[u]=\sum_{j=1}^{m}
\left( \int_{\Omega}\hat{\alpha}_{jk}(x)u_j(x)\,dx + \sum_{l=1}^{n}\int_{\Omega}\hat{\beta}_{jkl}(x)\de_{x_l} u_j(x)\,dx
\right)
\end{equation}
here, in \eqref{multipoint-intro}, $\hat{\alpha}_{ijk}, \hat{\beta}_{ijkl}$ are non-negative coefficients and $\omega_i, \tau_i\in \OO$ while, in \eqref{integral-intro}, $\hat{\alpha}_{jk}, \hat{\beta}_{jkl}$ are non-negative continuous functions on $\overline{\OO}$.
In particular we observe that nonlinear, nonlocal BCs have seen recently attention in the framework of elliptic equations: 
we refer the reader to the papers~\cite{genupa, genupa2, Goodrich3, Goodrich4, gi-tmna, gi-jepe, Pao-Wang} and references therein.

We wish to point out that an advantage of our setting, with respect to the theory developed in~\cite{genupa, genupa2, Goodrich3, Goodrich4, gi-tmna, Pao-Wang}, is the possibility to allow also gradient dependence within the functionals occurring in the BCs. This follows the approach used recently in~\cite{gi-nieto, gi-ho} within the setting of ODEs. 

Note that functional BCs that involve gradient terms may occur in applications. 
For example, consider a particular case of~\eqref{nellbvp-introduction} for $m=1$ and $n=2$, namely
\begin{equation} \label{gradient-intro}
 \left\{
\begin{array}{lll}
 -\Delta u(x)=f (x,u(x), \nabla u(x)), &  \text{$x \in B$} & \\
 u(x) = \eta_0 u(0) + \eta_1 \|\nabla u(0)\|, &  \text{$x \in \de B$} &
\end{array}
\right.
\end{equation}
 where $B$ is the Euclidean ball in $\R^2$ centered at $0$ with radius $1$, $\|\cdot\|$ is the Euclidian norm and  $\eta_i$ are non-negative coefficients. The BVP~\eqref{gradient-intro} can be used as a model for the steady states of the temperature of a heated disk of radius 1, where a controller located in the border of the disk adds or removes heat
according to the value of the temperature and to its variation, both registered by a sensor located in the center of the disk.
In the context of ODEs, a good reference for this kind of thermostat problems is the recent paper \cite{Webb}.

As already pointed out, a peculiarity of system \eqref{nellbvp-introduction} is the dependence on the gradient of the solutions, both in the nonlinearity and in the functionals occurring in the BCs,
and this represents the main technical difficulty that we have to deal with in this paper.
For this purpose, we have to perform a preliminary study of the Green's function 
of the partial differential operators which occur in \eqref{nellbvp-introduction}.
In Section \ref{sec.preliminaries} we collect some properties and estimates on Green's function, which are probably known to the experts in the field, nevertheless we include them for the sake of completeness.
Roughly speaking, these estimates yield the a priori bounds needed to compute the fixed point index in suitable cones of non-negative functions.

Section \ref{sec.existenceenon} contains our main results, while the final Section \ref{sec.examples} includes
some examples illustrating our results. In particular, we fix $m=2$ and $n=3$, and,
 taking into account the parameters $\lambda_1, \lambda_2, \eta_1,\eta_2$,
we provide existence and non-existence results
in some concrete situations.

\section{Preliminaries on divergence-form elliptic operators} \label{sec.preliminaries}
 In this Section we present, mostly without proof, several results
 concerning divergence-form operators which shall
 play a central r\^ole in the forthcoming sections.
 We refer the reader
 to, e.g., \cite{Evans, GT} for a detailed treatment of this topic. \medskip
 
 To being with, let $\mathcal{O}\subseteq\R^n$ be a fixed open set
 and let $\LL$ be a second-order linear PDO
 on $\mathcal{O}$ of the following divergence form:
 \begin{equation} \label{eq.generalformLL}
  \begin{split}
  \LL u & := 
   -\sum_{i,j = 1}^n
  \de_{x_i}\Big(a_{i,j}(x)\de_{x_j}u + b_i(x)u\Big)+
  \sum_{i = 1}^n c_i(x)\de_{x_i}u + d(x)u \\
  & = -\mathrm{div}\Big(A(x)\nabla u
  + \mathbf{b}u\Big) + \langle \mathbf{c},\nabla u\rangle + du
  \end{split}
 \end{equation}
 (here, $\mathbf{b} = (b_1,\ldots,b_n)$ and $\mathbf{c} = (c_1,\ldots,c_n)$).
 Throughout the sequel, we shall suppose
 that the following \textquotedblleft structural assumptions''
 on $\mathcal{O}$ and $\LL$ are satisfied:
 \begin{itemize}
  \item[(H0)] $\mathcal{O}$ is bounded, connected and of class $C^{1,\alpha}$ for some
  $\alpha\in (0,1)$;
  \item[(H1)] the coefficient functions of $\LL$ are H\"older-continuous
  of exponent $\alpha$
  up to $\de\mathcal{O}$, i.e.,
  \begin{equation*} 
   a_{i,j},\,b_i,\,c_i,\,d\in C^{\alpha}(\overline{\mathcal{O}},\R) \qquad\text{for
   every $i,j\in\{1,\ldots,n\}$};
  \end{equation*}
  \item[(H2)] the matrix $A(x) := \big(a_{i,j}(x)\big)_{i,j}$ is symmetric in $\mathcal{O}$, i.e.,
  $$a_{i,j}(x) = a_{j,i}(x) \qquad \text{for every $x\in\overline{\mathcal{O}}$
  and every $i,j\in\{1,\ldots,n\}$};$$
  \item[(H3)] $\LL$ is uniformly elliptic in $\mathcal{O}$, i.e., there exists
  $\Lambda > 0$ such that
  $$\frac{1}{\Lambda}\|\xi\|^2\leq \sum_{i,j = 1}^na_{i,j}(x)\xi_i\xi_j \leq \Lambda\|\xi\|^2\qquad
  \text{for any $x\in\overline{\mathcal{O}}$ and any $\xi\in\R^n$};$$
  \item[(H4)] the inequalities
   $d-\mathrm{div}(\mathbf{b})\geq 0$ and $d-\mathrm{div}(\mathbf{c}) \geq 0$ 
  hold in the weak sense of
  di\-stri\-bu\-tions on $\mathcal{O}$, i.e., for every $\varphi\in C_0^\infty(\mathcal{O},\R)$
  such that $\varphi\geq 0$ on $\mathcal{O}$, one has
  $$\int_{\mathcal{O}}
  \Big(d\varphi+{\textstyle\sum_{i = 1}^n}b_i\de_{x_i}\varphi\Big)\,\d x
  \geq 0\quad\text{and}\quad
  \int_{\mathcal{O}}
  \Big(d\varphi+{\textstyle\sum_{i = 1}^n}c_i\de_{x_i}\varphi\Big)\,\d x\geq 0.$$
 \end{itemize}
 It should be noticed that, since the coefficient functions of $\LL$ are assumed to be 
 just H\"older-continuous on $\overline{\mathcal{O}}$, it is not possible
 to compute $\LL u$ in a point-wise sense (even if $u$ is smooth on $\mathcal{O}$); for this reason,
 the following definition is plainly justified.
 \begin{defn} \label{def.weaksolLLuf}
  Let the assumptions (H0)-to-(H4) be in force, and let
  $f\in L^2(\mathcal{O})$.
  We say that a function
  $u:\mathcal{O}\to\R$ is a \emph{solution of the equation}
 \begin{equation} \label{eq.equationLuf}
  \LL u = f \quad \text{in $\OO$},
 \end{equation}
 if $u\in W^{1,2}(\mathcal{O})$ and if, for every test function
 $\phi\in C_0^\infty(\mathcal{O},\R)$, one has
  $$\int_\mathcal{O}
  \Big(\langle A(x)\nabla u + \mathbf{b}u,\nabla \phi\rangle +
  \langle \mathbf{c}, \nabla u\rangle \phi + du\phi\Big)\d x =
  \int_{\mathcal{O}}f\phi\,\d x.$$
  Given $g\in W^{1,2}(\mathcal{O})$, we say that $u$ is a 
  \emph{solution of the Poisson problem}
  \begin{equation} \label{eq.PossonL}
   \begin{cases}
    \LL u = f & \text{in $\OO$}, \\
    u\big|_{\de\OO} = g,
   \end{cases}
  \end{equation}
  if $u$ is a solution of \eqref{eq.equationLuf}
  and, furthermore, $u-g\in W_0^{1,2}(\mathcal{O})$.
 \end{defn}
 Now, as a consequence
 of the \textquotedblleft sign assumption'' (H4) it is possible to prove that
 a suitable form of the Weak Maximum Principle holds for $\LL$
 (see, e.g., \cite[Theorem 8.1]{GT}); from this, one can
 straight\-for\-war\-dly deduce
  Lemma \ref{lem.uniquePoisson} below (see \cite[Corollary 8.2]{GT}), ensuring 
  that the Poisson problem
 \eqref{eq.PossonL} possesses at most one solution.
 \begin{lem} \label{lem.uniquePoisson}
  Let the assumptions 
  \emph{(H0)}-to-\emph{(H4)} be in force, and let $u\in W_0^{1,2}(\mathcal{O})$
  be such that $\LL u = 0$ or $\LL^T u = 0$ in $\mathcal{O}$. 
  Then $u \equiv 0$ almost everywhere 
  on $\mathcal{O}$.
 \end{lem}
 
 \subsection{The Poisson problem for $\LL$} \label{subsec:Poisson}
  A first group of results we
  aim to present is about existence and regularity of solutions
  for the Poisson problem \eqref{eq.PossonL} for $\LL$.
   In order to do this, we first introduce the following Banach spaces:
  \begin{itemize}
   \item $X = (C(\clOO,\R), \|\cdot\|_\infty)$, where 
   \begin{equation} \label{eq.defnormainfty}
   \|f\|_\infty 
   := \max_{x\in\clOO}|f(x)|;
   \end{equation}
   \item $X = (C^1(\clOO,\R), \|\cdot\|_{C^1(\clOO,\R)})$, where
   \begin{equation} \label{eq.defnormaC1scalar}
    \|f\|_{C^1(\clOO,\R)} := \max_{j = 1,\ldots,n}
   \big\{\|f\|_\infty, \|\de_jf\|_\infty:\,j = 1,\ldots,n\big\};
   \end{equation}
   \item $X = C^{1,\theta}(\clOO,\R)$ (for some
  $\theta\in(0,1)$), where
   \begin{equation} \label{eq.defnormaC1thetascalar}
    \|u\|_{C^{1,\theta}(\clOO,\R)} := \max_{j = 1,\ldots,n}\Big\{\|u\|_\infty,\,
  \|\de_{j}u\|_\infty,\,\sup_{x,y\in\clOO}\frac{|\de_{j}u(x)-\de_j u(y)|}{\|x-y\|^\theta}
  \Big\}
  \end{equation}
  \end{itemize}
Given $f\in C^1(\clOO,\R)$, it will be also convenient to define, with abuse of notation,
\begin{equation} \label{eq.defnormagrad}
    \|\nabla f\|_{\infty} := \max_{j = 1,\ldots,n}
   \big\{\|\de_jf\|_\infty:\,j = 1,\ldots,n\big\},
   \end{equation}
so that, clearly, $\|f\|_{C^1(\clOO,\R)} = \max \big\{\|f\|_\infty, \|\nabla f\|_\infty\big\}$.
  \medskip  

  Now, by exploiting assumptions (H3)-(H4),
  Lemma \ref{lem.uniquePoisson}
  and the Fredholm al\-ter\-na\-ti\-ve, one can establish
  the following basic theorem
  (for a proof, see \cite[Theorem 8.3]{GT}).
  \begin{thm} \label{thm.existenceFirst}
   Let the assumptions
   \emph{(H0)}-to-\emph{(H4)} be in force.  
   Then, for every $f\in L^2(\mathcal{O})$ and every $g\in W^{1,2}(\mathcal{O})$
   there exists a unique solution $u_{f,\,g}\in W^{1,2}(\mathcal{O})$
   of \eqref{eq.PossonL}.
  \end{thm}
  Throughout the sequel, we indicate by $u_{f,\,g}$ the unique solution
  in $W^{1,2}(\mathcal{O})$ of \eqref{eq.PossonL}
  (for fixed $f\in L^2(\mathcal{O})$ and $g\in W^{1,2}(\mathcal{O})$),
  whose existence is guaranteed
  by Theorem \ref{thm.existenceFirst}.
  In the particular case when $g \equiv 0$, we simply write $u_f$
  instead of $u_{f,\,0}$.
  \begin{rem} \label{rem.hypothesisLemmaExistence}
   Theorem \ref{thm.existenceFirst}
   holds under
   more general hypotheses: in fact, it suffices to assume that 
   {$\mathcal{O}$} is bounded and that the coefficient functions of $\LL$ are
   in $L^\infty({\mathcal{O}})$.
  \end{rem}
  \begin{rem} \label{rem.linearKKA}
  Let $f_1,f_2\in L^2({\mathcal{O}})$ and, for $i = 1,2$, let
  $u_i = u_{f_i}\in W_0^{1,2}({\mathcal{O}})$ be the unique solution
  of \eqref{eq.PossonL} with $f = f_i$ (and $g\equiv 0$). Since, obviously, it holds that
  $$\LL (u_{f_1}+u_{f_2}) = u_{f_1}+u_{f_2}\quad\text{and}\quad
  u_{f_1}+u_{f_2}\in W_0^{1,2}({\mathcal{O}}),$$
  we conclude that the unique solution of \eqref{eq.PossonL} with $f = f_1+f_2$
  and $g\equiv 0$
  is $u_{f_1}+u_{f_2}$.
 \end{rem}
  Since we aim to apply suitable fixed-point techniques
  to operators acting on spaces of $C^1$-functions, we are interested
  in solving \eqref{eq.PossonL} for {continuous} $f$
  and regular $g$. In this context, the unique solution
  $u_{f,\,g}$ of \eqref{eq.PossonL} turns out to be much more regular that $W^{1,2}$;
  in fact, we have the following crucial 
  result (for a proof, see
   \cite[Thm.s 8.16, 8.33 and 8.34]{GT}).
  \begin{thm} \label{thm.regulConegamma}
  Let the assumptions \emph{(H0)}-to-\emph{(H4)}, and let
  $\LL$ be as in \eqref{eq.generalformLL}. Moreover,
   let $f\in C(\overline{\OO},\R)$ and let $g\in C^{1,\alpha}(\overline{\OO},\R)$.
  Then the following facts hold true.
  \begin{itemize}
   \item[{(i)}] There exists a unique $\hat{u}_{f,\,g}\in C^{1,\alpha}(\overline{\OO},\R)$
   such that 
   $$\text{$\hat{u}_{f,\,g} \equiv u_{f,\,g}$ a.e.\,on $\OO$}.$$ 
   In particular, $\hat{u}_{f,\,g}$ 
   solves \eqref{eq.equationLuf} and $\hat{u}_{f,\,g}\equiv g$ point-wise on $\de\Omega$.
   
   \item[(ii)] There exists a constant $C > 0$, 
   only depending on $n,\,\Lambda$ and $\mathcal{O}$, such that
   \begin{equation} \label{eq.estimuff}
    \|\hat{u}_{f,\,g}\|_{C^{1,\alpha}(\overline{\OO},\R)}
    \leq C\,\Big(\|f\|_{C(\overline{\OO},\R)}+\|g\|_{C^{1,\alpha}(\overline{\OO},\R)}\Big).
   \end{equation}
   
   \item[(iii)] If $f\geq 0$ on $\overline{\OO}$ and
   $g\geq 0$ on $\de\OO$, then $\hat{u}_{f,\,g}\geq 0$
   on $\overline{\OO}$.
  \end{itemize}
 \end{thm}
 Now, in view of Theorem \ref{thm.regulConegamma}-(i), we can define
 a linear operator as follows
 \begin{equation} \label{eq.defoperatorG}
  \GG_\LL: C(\overline{\OO},\R)\longrightarrow C^{1,\alpha}(\overline{\OO},\R), \qquad
  \GG_\LL(f) := \hat{u}_f, 
 \end{equation}
 where $\hat{u}_f = \hat{u}_{f,\,0}\in C^{1,\alpha}(\overline{\OO},\R)$
 is the unique solution of \eqref{eq.PossonL} with $g\equiv 0$. We shall call
 $\GG_\LL$ the \emph{Green operator for $\LL$}. 
 By exploiting assertions (ii)-(iii) of
  Theorem \ref{thm.regulConegamma},
  it is possible to deduce some continuous-compactness properties
  of $\GG_\LL$ which shall play a central r\^ole in the next
  sections; to be more precise, we have the following proposition. 
  \begin{pro} \label{prop.propertiesG}
   Let the assumptions \emph{(H0)}-to-\emph{(H4)} be in force, and let
   $\GG_\LL$ be the operator defined in \eqref{eq.defoperatorG}.
   Then the following facts hold:
   \begin{itemize}
    \item[(i)] $\GG_\LL$ is continuous
    from $C(\clOO,\R)$ to  $C^{1,\alpha}(\clOO,\R)$;
    \item[(ii)] $\GG_\LL$ is \emph{compact}
    from $C(\clOO,\R)$ to  $C^{1}(\clOO,\R)\supseteq
    C^{1,\alpha}(\clOO,\R)$;
    \item[(iii)] if $V_0:=C(\clOO,\R^+)
    \subseteq C(\clOO,\R)$ denotes the 
    \emph{(}convex\emph{)} cone of
    the non-negative con\-ti\-nuo\-us functions on $\clOO$, 
    it holds that $\GG_\LL(V_0)\subseteq V_0$.
   \end{itemize}
  \end{pro}
  \begin{proof}
  (i)\,\,On account of Theorem \ref{thm.regulConegamma}-(ii), 
  for every
  $f\in C(\clOO,\R)$ one has
  \begin{equation} \label{eq.continuityKA}
   \|\GG_\LL(f)\|_{C^{1,\alpha}(\clOO,\R)}
  \leq C\|f\|_{\infty}
  \end{equation}
  (here, $C > 0$ is a constant independent of $f$).
  Since $\GG_\LL$ is linear (see Remark \ref{rem.linearKKA}), 
  from \eqref{eq.continuityKA} we immediately deduce that
  $\GG_\LL$ is continuous
  from $C(\clOO,\R)$ to  $C^{1,\alpha}(\clOO,\R)$. \medskip
  
  (ii)\,\,Let $\{f_j\}_{j}$ be a bounded sequence in $C(\clOO,\R)$.
  On account of \eqref{eq.continuityKA}, we see that the sequence $\{\GG_\LL(f_j)\}_j$
  is bounded in $C^{1,\alpha}(\clOO,\R)$; as a consequence,
  a standard application of Arzelà-Ascoli's Theorem implies
  the existence of $u_0,\ldots,u_n\in C(\clOO,\R)$ such that \medskip
  
  (a)\,\,$\|\GG_\LL(f_{j_k})- u_0\|_{\infty}\to 0$ as $k\to\infty$, \medskip
  
  (b)\,\,$\|\de_{i}(\GG_\LL(f_{j_k}))- u_i\|_{\infty}\to 0$
  as $k\to\infty$ (for every $i = 1,\ldots,n$), \medskip
  
  \noindent where $\{f_{j_k}\}_{k}$ is a suitable sub-sequence of
  $\{f_{j}\}_j$. By combining (a) and (b), we deduce that
  $u_0\in C^1(\clOO,\R)$ and that $\nabla u_0 = (u_1,\ldots,u_n)$; moreover,
  one has
  $$\|\GG_\LL(f_{j_k})- u_0\|_{C^1(\clOO),\R)}\to 0\quad \text{as $k\to\infty$},$$ 
  and this proves that $\GG_\LL$ is compact from 
  $C(\clOO,\R)$ to  $C^{1}(\clOO,\R)$, as desired. \medskip
    
  (iii)\,\,Let $f\in V_0$ be fixed. Since, by Theorem \ref{thm.regulConegamma}-(iii),
  we know that $\GG_\LL(f) = \hat{u}_{f,\,0}\geq 0$ throug\-hout $\clOO$, we immediately conclude that
  $\GG_\LL(f)\in V_0\cap C^{1,\alpha}(\clOO,\R)$, as desired.
  \end{proof}
  \subsection{Green's function for $\LL$} \label{subsec:Greenfunction}
  Now we have established Proposition \ref{prop.propertiesG}, we turn to present
  a second group of results: this is about
  the existence of a Green's function for $\LL$ allowing to obtain an integral representation
  formula for $\GG_\LL$. \medskip
  
  To begin with, we demonstrate the 
  following key
  theorem.
  \begin{thm} \label{thm.mainexistenceGreen}
   Let the assumptions \emph{(H0)}-to-\emph{(H4)} be in force, and let
   $\LL$ be as in \eqref{eq.generalformLL}. There exists a
   function $g_\LL:\mathcal{O}\times\mathcal{O}\to[0,\infty)$ such that
   \begin{itemize}
    \item[(a)] $g_{\LL}(\cdot;x)\in L^1(\OO)$ for almost every $x\in\OO$;
    \item[(b)] for every $f\in C(\clOO,\R)$ one has
    \begin{equation} \label{eq.Greenfunction}
     \GG_\LL(f)(x) = \int_{\OO}g_\LL(y;x)f(y)\,\d y \qquad\text{for a.e.\,$x\in\OO$}.
    \end{equation}
   \end{itemize}
   Furthermore, $g_\LL$ enjoys the following properties:
    \begin{itemize}
     \item[(I)] there exists a constant $c_0 > 0$ such that, for a.e.\,$x,y\in\OO$, one has
     \begin{equation} \label{eq.mainestimG}
      0\leq g_\LL(y;x)\leq c_0\,\|x-y\|^{2-n};
     \end{equation}
     \item[(II)] $g_\LL(\cdot;x)\in
     W_0^{1,p}(\mathcal{O})$ for a.e.\,$x\in\OO$ 
    and every $1\leq p < n/(n-1)$;
    \item[(III)] $g_\LL(y;\cdot)\in
     W_0^{1,p}(\mathcal{O})$ for a.e.\,$y\in\OO$ 
    and every $1\leq p < n/(n-1)$;
     \item[(IV)] there exists a constant $c_1 > 0$ such that,
     for a.e.\,$x,y\in\OO$, one has
     \begin{equation} \label{eq.mainestimderG}
      \|\nabla_y g_\LL(y;x)\|\leq c_1\,\|x-y\|^{1-n}\quad\text{and}\quad
      \|\nabla_x g_\LL(y;x)\|\leq c_1\,\|x-y\|^{1-n}.
     \end{equation}
    \end{itemize}
    Finally, $g_\LL$ is unique in the following sense: if $\tilde{g}:\OO\times\OO\to[0,\infty)$
    is another function satisfying \emph{(a)-(b)}, then $g_\LL(\cdot;x) = \tilde{g}(\cdot;x)$
    in $L^1(\OO)$ for a.e.\,$x\in\OO$.
  \end{thm}
  Throughout the sequel, we shall refer to the function
   $g_\LL$ in Theorem \ref{thm.mainexistenceGreen} as the
  \emph{Green's function for the operator $\GG_\LL$ \emph{(}and 
  related to the open set $\OO$\emph{)}}.
  \begin{proof}
   We begin by proving the existence part of the theorem. In order to do this,
   we make pivotal use of several results established in the very recent paper
   \cite{KimSak}. \medskip
   
   \noindent First of all, by \cite[Proposition 5.3]{KimSak} there exists
   a function $g_\LL:\OO\times\OO\to\R$ such that \medskip
   
   (i)\,\,$g_\LL(\cdot;x)\in W^{1,p}(\OO)$ for a.e.\,$x\in\OO$ and every
   $1\leq p < n/(n-1)$; \medskip
   
   (ii)\,\,for every fixed $f\in C(\clOO,\R)$ one has
   $$\GG_\LL(f)(x) = \int_{\OO}g_\LL(y;x)f(y)\,\d y \qquad\text{for a.e.\,$x\in\OO$}.$$
   Moreover, by \cite[Theorem 6.10]{KimSak} we also have that
   $$0 \leq g_\LL(y;x)\leq c_0\,\|x-y\|^{2-n} \qquad\text{for a.e.\,$x,y\in\OO$ with $x\neq y$},$$
   where $c_0 > 0$ is a suitable constant. In view of these facts, to complete the demonstration
   we
   are left to prove assertion (iii) and the point-wise estimates
   in \eqref{eq.mainestimderG}. \vspace*{0.08cm}
   
   To this end, let us introduce the so-called (formal) adjoint $\LL^T$
   of $\LL$: this is the linear differential operator defined on $\OO$ in the following way
   \begin{equation} \label{eq.generalformLLT}
  \begin{split}
  \LL^T v & := 
   -\sum_{i,j = 1}^n
  \de_{x_i}\Big(a_{i,j}(x)\de_{x_j}v + c_i(x)v\Big)+
  \sum_{i = 1}^n b_i(x)\de_{x_i}v + d(x)v \\
  & = -\mathrm{div}\Big(A(x)\nabla v
  + \mathbf{c}v\Big) + \langle \mathbf{b},\nabla v\rangle + dv.
  \end{split}
 \end{equation}
 Clearly, $\LL^T$ takes the same divergence-form
 of $\LL$ in \eqref{eq.generalformLL} (with $\mathbf{b}$ and $\mathbf{c}$
 interchanged); furthermore, due to the \textquotedblleft symmetry''
 in assumption (H4), it is readily seen that $\LL^T$ satisfies 
 the \textquotedblleft structural assumptions'' (H1)-to-(H4).
 
 As a consequence, all the results established so far
 do apply to $\LL^T$. 
  In particular, 
 for every fixed $g\in C(\clOO,\R)$ there exists a unique
 function $\mathcal{T}(g)\in C^{1,\alpha}(\clOO,\R)$ such that
 $$\LL^T\mathcal{T}(g) = g\quad \text{in $\OO$}\qquad\quad\text{and}\qquad\quad
 \mathcal{T}(g) \equiv 0\quad\text{on $\de\OO$}.$$ 
 Now, by \cite[Theorem 6.12]{KimSak} there exists a function $G:\OO\times\OO\to\R$ such that
 \medskip
   
   (iii)\,\,$G(\cdot;y)\in W^{1,p}(\OO)$ for a.e.\,$y\in\OO$ and every
   $1\leq p < n/(n-1)$; \medskip
   
   (iv)\,\,for every fixed $g\in C(\clOO,\R)$ one has
   $$\mathcal{T}(g)(y) = \int_{\OO}G(x;y)g(x)\,\d x \qquad\text{for a.e.\,$y\in\OO$}.$$
   On the other hand, since \cite[Proposition 6.13]{KimSak} shows that
   \begin{equation} \label{eq.relationGgsymmetry}
    G(x;y) = g_\LL(y;x) \qquad\text{for a.e.\,$x,y\in\OO$ with $x\neq y$},
    \end{equation}
   from (iii) we infer that 
   $g_\LL(y;\cdot) = G(\cdot;y)\in W^{1,p}(\OO)$ for almost every $y\in\OO$ and every
   exponent $p\in [1,n/(n-1))$. This is exactly assertion (III).
   
   Finally, we prove the point-wise estimates in assertion (IV). 
   First of all, since $\LL$ satisfies assumptions (H1)-to-(H4), we are entitled to apply
   \cite[Theorem 8.1]{KimSak}, ensuring that
   \begin{equation} \label{eq.estimderg1}
    \|\nabla_x G(x;y)\|\leq c_1'\,\|x-y\|^{1-n}\qquad\text{for a.e.\,$x,y\in\OO$ with $x\neq y$},
    \end{equation}
   where $c'_1 > 0$ is a suitable constant. Moreover, since
   also $\LL^T$ satisfies assumptions (H1)-to-(H4), another application
   of \cite[Theorem 8.1]{KimSak} gives
   \begin{equation} \label{eq.estimderg2}
    \|\nabla_y g_\LL(y;x)\|\leq c_1''\,\|x-y\|^{1-n}\qquad\text{for a.e.\,$x,y\in\OO$
   with $x\neq y$},
   \end{equation}
   where $c_1'' > 0$ is another suitable constant. Gathering together
   \eqref{eq.estimderg2}, \eqref{eq.estimderg1} and \eqref{eq.relationGgsymmetry} we immediately
   obtain the desired \eqref{eq.mainestimderG} (with $c_1 := \max\{c_1',c_1''\}$). \medskip
   
    As for the uniqueness part of the theorem, let us suppose that there exists
    another function $\tilde{g}:\OO\times\OO\to[0,\infty)$ satisfying (a)-(b).
    In particular, for every $\phi\in C_0^\infty(\OO,\R)$ one has
    \begin{equation} \label{eq.gLLgzero}
     \int_{\OO}\big(g_\LL(y;x)-\tilde{g}(y;x)\big)\phi(y)\,\d y = 0 \quad
    \text{for a.e.\,$x\in\OO$}.
    \end{equation}
    Now, the space $C_0^\infty(\OO,\R)$ being separable
    (with its usual LF-topology), there exists a countable
    set $\mathcal{F}\subseteq C_0^\infty(\OO,\R)$ which is dense;
    moreover, by \eqref{eq.gLLgzero}, for every $\phi\in\mathcal{F}$ there
    exists a set $E(\phi)\subseteq \OO$, with zero-Lebesgue measure, such that
    $$\int_{\OO}\big(g_\LL(y;x)-\tilde{g}(y;x)\big)\phi(y)\,\d y = 0 \quad
    \text{for all $x\in E(\phi)$}.$$
    We then define $E:=\cup_{\phi\in\mathcal{F}}E(\phi)$. Since $\mathcal{F}$
    is countable and $E(\phi)$ has zero-Lebesgue measure for every $\phi$,
    we see that $E$ has measure zero; moreover,
    for every $x\in\OO\setminus E$ we have
    $$\int_{\OO}\big(g_\LL(y;x)-\tilde{g}(y;x)\big)\phi(y)\,\d y = 0 \quad
    \text{for all $\phi\in\mathcal{F}$}.$$
    This proves that, for every $x\in\OO\setminus E$, the distribution
    $g_\LL(\cdot;x)-\tilde{g}(\cdot;x)$ vanishes on $\mathcal{F}$; the latter being
    dense, we then conclude that
    $g_\LL(\cdot;x) = \tilde{g}(\cdot;x)$ in $L^1(\OO)$ for a.e.\,$x,y\in\OO$.
    
    This ends the proof.
  \end{proof}
  \begin{rem} \label{rem.onassumptionH4}
   The approach adopted for the proof of Theorem \ref{thm.mainexistenceGreen}
   shows the reason why we have assumed that
   $d-\mathrm{div}(\mathbf{b})\geq 0$ and $d-\mathrm{div}(\mathbf{c})\geq 0$
   in the sense of distributions.
   
   In fact, under this assumption, 
   all the mentioned results in \cite{KimSak} hold both for $\LL$ and for its transpose
   $\LL^T$; in particular, this allows us to obtain point-wise estimates
    both for 
    $$\text{$\nabla_x g_\LL(y;x) = \nabla_x G(x;y)$ \qquad and \qquad
   $\nabla_y g_\LL(y;x)$}.$$
  \end{rem}
  \begin{rem} \label{rem.gLLsymmetric}
   It is contained in the proof of Theorem \ref{thm.mainexistenceGreen}
   the following fact: if $\LL$ is of the form \eqref{eq.generalformLL}
   and if $\mathbf{b} \equiv \mathbf{c}$ on $\OO$, then the 
   Green's function for $\GG_\LL$ is symmetric, that is,
   $$g_\LL(y;x) = g_\LL(x;y) \qquad\text{for a.e.\,$x,y\in\OO$}.$$
   In fact, if 
   $\mathbf{b} \equiv \mathbf{c}$ on $\OO$, then the adjoint operator
   $\LL^T$ coincides with
   $\LL$ (see \eqref{eq.generalformLLT}); thus, following the notation
   in the proof of Theorem \ref{thm.mainexistenceGreen}, we have
   $$g_\LL(x;y) = G(x;y) = g_\LL(y;x).$$
  \end{rem}
  \begin{rem} \label{rem.regulgLL}
   By carefully scrutinizing the proofs of the existence results for $g_\LL$ contained
   in \cite[Proposition 5.3]{KimSak}, 
   one can recognize that the following properties hold: \medskip

   (a)\,\,for a.e.\,$x\in\OO$ and every
   $\epsilon > 0$, we have
   $g_\LL(\cdot;x)\in W^{1,2}(\OO\setminus B(x,\epsilon))$; 

  (b)\,\,$g_\LL(\cdot;x)$ is a solution of $\LL^T u = 0$ in $\OO\setminus B(x,\epsilon)$,
  where $\LL^T$ is as in \eqref{eq.generalformLLT}. \medskip
  
  \noindent Analogously, an inspection to the proof
  of \cite[Theorem 6.12]{KimSak} shows that \medskip

   (a')\,\,for a.e.\,$y\in\OO$ and every
   $\epsilon > 0$, we have
   $G(\cdot;y) = g_\LL(y;\cdot)\in W^{1,2}(\OO\setminus B(y,\epsilon))$; 

  (b')\,\,$G(\cdot;y) = 
  g_\LL(y;\cdot)$ is a solution of $\LL u = 0$ in $\OO\setminus B(y,\epsilon)$. \medskip
  
  \noindent Gathering together all these facts, from the classical elliptic
  regularity theory (see, e.g., \cite[Corollary 8.36]{GT}) 
  we deduce that $g_\LL$ is of class $C^{1,\alpha}$
  out of the diagonal of $\OO\times\OO$. 
  \end{rem}
  We now use the point-wise estimates in \eqref{eq.mainestimG}-\eqref{eq.mainestimderG}
  to prove the following lemma.
  \begin{lem} \label{lem.propertiesGLintegral}
  Let the assumptions \emph{(H0)}-to-\emph{(H4)} be in force, and let
  $g_\LL$ be the Green's function for $\GG_\LL$. Moreover, let $\rho :=
  \mathrm{diam(\OO)}$. Then, the following estimates hold:
  \begin{align}
   & \int_{\OO}g_\LL(y;x)\,\d y \leq c_0\cdot\frac{n\,\omega_n\,\rho^2}{2} 
    \qquad\text{for a.e.\,$x\in\OO$}; \label{eq.estimGintegral} \\[0.2cm]
   & \int_{\OO}\big|\de_{x_i}g_\LL(y;x)\big|\,\d y \leq c_1\cdot n\,\omega_n\,\rho.
   \qquad\text{for a.e.\,$x\in\OO$}. \label{eq.estimderGintegral}
  \end{align}
  Here, $\omega_n$ is the Lebesgue measure of the unit ball $B(0,1)\subseteq\R^n$.
  \end{lem}
  \begin{proof}
   We begin by proving \eqref{eq.estimGintegral}. To this end we first notice
   that, if $x\in\OO$ is arbitrary, then $\OO\subseteq \overline{B(x,\rho)}$; 
   as a consequence, by crucially exploiting estimate \eqref{eq.mainestimG}
   we get
   \begin{align*}
    \int_{\OO}g_\LL(y;x)\,\d y & \leq c_0\,\int_{\OO}\|x-y\|^{2-n}\,\d y
    \leq c_0\,\int_{\overline{B(x,\rho)}}\|x-y\|^{2-n}\,\d y \\
    & = c_0\,\int_{\overline{B(0,\rho)}}\|y\|^{2-n}\,\d y 
    = c_0\,\int_{0}^\rho t^{2-n}\,\mathcal{H}^{n-1}(\de B(0,t))\,\d t \\
    & = c_0\,n\,\omega_n\,\int_0^\rho t\,\d t = c_0\cdot\frac{n\,\omega_n\,\rho^2}{2},
   \end{align*}
   which is exactly the desired \eqref{eq.estimGintegral}. As for the proof
   of \eqref{eq.estimderGintegral}, we argue essentially in the same way:
   by crucially exploiting the estimate~\eqref{eq.mainestimderG}
   we get 
   \begin{align*}
    \int_{\OO}\big|\de_{x_i}g_\LL(y;x)\big|\,\d y & \leq c_1\,\int_{\OO}\|x-y\|^{1-n}\,\d y
    \leq c_1\,\int_{\overline{B(x,\rho)}}\|x-y\|^{1-n}\,\d y \\
    & = c_1\,\int_{\overline{B(0,\rho)}}\|y\|^{1-n}\,\d y 
    = c_1\,\int_{0}^\rho t^{1-n}\,\mathcal{H}^{n-1}(\de B(0,t))\,\d t \\
    & = c_1\,n\,\omega_n\,\int_0^\rho\d t = c_1\cdot n\,\omega_n\,\rho,
   \end{align*}
   and this is precisely the desired inequality~\eqref{eq.estimderGintegral}.
  \end{proof}
  \begin{rem} \label{rem.estimateGGLL1}
   We explicitly observe that, by combining the estimate~\eqref{eq.estimGintegral}
   in Lemma~\ref{lem.propertiesGLintegral} with the representation
   formula~\eqref{eq.Greenfunction}, for a.e.\,$x\in\OO$ we obtain 
   $$0\leq \GG_\LL(\hat{1})(x) = \int_{\OO}g(y;x)\,\d y \leq c_0\cdot\frac{n\,\omega_n\,\rho^2}{2},$$
   where $\rho := \mathrm{diam}(\OO)$
   and $\hat{1}$ denotes the 
   constant function equal to  $1$ on $\OO$. As a consequence, since
   $\GG_\LL(\hat{1})\in C(\clOO,\R)$, we get
   $$\|\GG_\LL(\hat{1})\|_{\infty}\leq  c_0\cdot\frac{n\,\omega_n\,\rho^2}{2}.$$
  \end{rem}
  We conclude this part of the Section by deducing from \eqref{eq.Greenfunction} 
  an integral representation
  for the $x_i$-derivatives of $\GG_\LL(f)$.
  To this end we first observe that, if $f\in C(\clOO,\R)$, 
  Lemma~\ref{lem.propertiesGLintegral} ensures that the following
  \textquotedblleft potential-type'' functions are well-defined:
  \begin{equation} \label{eq.potentialPi}
   \mathcal{P}^{(i)}f(x) := \int_{\OO}\de_{x_i}g_\LL(y;x)f(y)\,\d y \qquad 
	\big(\text{for $i = 1,\ldots,n$}\big). 
  \end{equation}
  In fact, by estimate \eqref{eq.estimderGintegral} in Lemma
  \ref{lem.propertiesGLintegral} we have (for $i = 1,\ldots,n$)
  \begin{align*}
   &  \int_{O}|\de_{x_i}g_\LL(y;x)|\cdot|f(y)|\,\d y
  \leq \|f\|_{\infty}\cdot\int_{\OO}|\de_{x_i}g_\LL(y;x)|\,\d y \\
  & \qquad\quad \leq \|f\|_{\infty}\cdot
  c_1\,n\,\omega_n\,\mathrm{diam}(\OO) \qquad\qquad (\text{for a.e.\,$x\in\OO$}).
  \end{align*}
  Moreover, from the above computation we also infer that (again for $i = 1,\ldots,n$) 
  $$\text{$\mathcal{P}^{(i)}f \in L^{\infty}(\OO)$\,\,\,and\,\,\,
  $\|\mathcal{P}^{(i)}f\|_{L^\infty(\OO)} \leq \|f\|_{\infty}\cdot
  c_1\,n\,\omega_n\,\mathrm{diam}(\OO)$}.$$
  We are then ready to prove the following Proposition.
  \begin{pro} \label{prop.Piweakderivative}
   Let the assumptions \emph{(H0)}-to-\emph{(H4)} be in force,
   and let $f\in C(\clOO,\R)$. Moreover, let $i\in\{1,\ldots,n\}$
   be fixed, and let $\mathcal{P}^{(i)}f$ be as in 
   \eqref{eq.potentialPi}. Then, we have
   \begin{equation} \label{eq.Piweakderivative}
    \de_{x_i}\GG_\LL(f)(x) = \mathcal{P}^{(i)}f(x) 
    = \int_{\OO}\de_{x_i}g_\LL(y;x)f(y)\,\d y \qquad\text{for a.e.\,$x\in\OO$}.
   \end{equation}
  \end{pro}
  \begin{proof}
   We first notice, since $\GG_\LL(f)\in C^{1,\alpha}(\clOO,\R)$, the identity
   \eqref{eq.Piweakderivative} follows if we show that the
   $L^\infty$-function $\mathcal{P}^{(i)}f$
   is the weak derivative (in $L^1(\OO)$) of $\GG_\LL(f)$.
   To prove this fact, we ar\-gue as follows: firstly, 
   if $\phi\in C_0^\infty(\OO,\R)$, by the
   estimate~\eqref{eq.estimGintegral} in Lemma \ref{lem.propertiesGLintegral} 
   we get
   \begin{align*}
   & \int_{\OO\times\OO}g_\LL(y;x)\cdot|f(y)|\cdot|\de_{x_i}\phi(x)|\,\d x\,\d y \\
   & \qquad\quad\leq \|f\|_{C^(\clOO,\R)}\cdot\|\de_{i}\phi\|_{\infty}\cdot\int_{\OO}
   \bigg(\int_{\OO}g_\LL(y;x)\,\d y\bigg)\,\d x \\
   & \qquad\quad \leq 
   \|f\|_{C^(\clOO,\R)}\cdot\|\de_{i}\phi\|_{\infty}\cdot
   c_0\cdot\frac{n\,\omega_n\,\mathrm{diam}(\OO)^2}{2}\cdot|\mathcal{O}|;
   \end{align*}
   we are then entitled
   to apply Fubini's Theorem, obtaining
   \begin{align*}
    & \int_{\OO}\GG_\LL(f)(x)\de_{x_i}\phi(x)\,\d x
     = \int_{\OO}\bigg(\int_{\OO}g_\LL(y;x)f(y)\,\d y\bigg)\de_{x_i}\phi(x)\,\d x \\
    & \qquad\quad = \int_{\OO}\bigg(\int_{\OO}g_\LL(y;x)\de_{x_i}\phi(x)\,\d x\bigg)
    f(y)\,\d y \\
    & \qquad\quad \big(\text{since $g_\LL(y;\cdot)\in W_0^{1,1}(\OO)$,
    see Theorem \ref{thm.mainexistenceGreen}-(III)}\big) \\
    & \qquad\quad 
    = -\int_{\OO}\bigg(\int_{\OO}\de_{x_i}g_\LL(y;x)\phi(x)\,\d x\bigg)
    f(y)\,\d y =: (\star).
   \end{align*}
   On the other hand, since the
   estimate~\eqref{eq.estimderGintegral} in Lemma \ref{lem.propertiesGLintegral} 
   implies that
   \begin{align*}
   & \int_{\OO\times\OO}|\de_{x_i}g_\LL(y;x)|\cdot|f(y)|\cdot|\phi(x)|\,\d x\,\d y \\
   & \qquad\quad\leq \|f\|_{\infty}\cdot\|\phi\|_{\infty}\cdot\int_{\OO}
   \bigg(\int_{\OO}|\de_{x_i}g_\LL(y;x)|\,\d y\bigg)\,\d x \\
   & \qquad\quad \leq 
   \|f\|_{\infty}\cdot\|\phi\|_{\infty}\cdot
   c_1\cdot n\,\omega_n\,\mathrm{diam}(\OO)\cdot|\mathcal{O}|,
   \end{align*}
   another application of Fubini's Theorem is legitimate, and we get
   \begin{align*}
    (\star) = -\int_{\OO}\bigg(\int_{\OO}\de_{x_i}g_\LL(y;x)f(y)\,\d y\bigg)
    \phi(x)\,\d x
    \stackrel{\eqref{eq.potentialPi}}{=}
    - \int_{\OO}\mathcal{P}^{(i)}f(x)\phi(x)\,\d x.
   \end{align*}
   Due to the arbitrariness of $\phi\in C_0^\infty(\clOO,\R)$, we then conclude
   that $\mathcal{P}^{(i)}f$ is the weak derivative of
   $\GG_\LL(f)$ in $L^1(\OO)$, and the proof is complete.
  \end{proof}
  \begin{rem} \label{rem.identityeverywhere}
   By using the regularity of $g_\LL$ described in Remark \ref{rem.regulgLL},
   it is quite standard to recognize that, for a fixed
   $f\in C(\clOO,\R)$, the functions
   $$\OO\ni x\mapsto \int_{\OO}g_\LL(y;x)f(y)\,\d y
   \qquad\text{and}\qquad\mathcal{P}^{(1)}f,\ldots,\mathcal{P}^{(n)}f$$
   are continuous on $\OO$. As a consequence, the representation formulas
   \eqref{eq.Greenfunction} and \eqref{eq.Piweakderivative} actually hold
   true
   \emph{for every $x\in\OO$} (not only almost everywhere).
  \end{rem}
  \subsection{Spectral properties of $\GG_\LL$} \label{subsec.spectrumGGLL}
   We conclude this section by briefly turning our attention
   to the spectral properties of the Green's operator $\GG_\LL$. \medskip
   
   To begin with, we remind the following theorem (see, e.g., \cite[Theorem 8.6]{GT}).
   \begin{thm} \label{thm.spectrumLL}
    Let the assumptions \emph{(H0)}-to-\emph{(H4)} be in force. Then, there exists
    a countable and discrete set $\Sigma\subseteq (0,\infty)$ with the following property:
    \emph{for every $\sigma\in\Sigma$ the subspace of solutions of the homogeneous
    problem}
    \begin{equation} \label{eq.eigenvalueLL}
     \begin{cases}
    \LL u = \sigma u & \text{in $\OO$}, \\
    u\big|_{\de\OO} = 0,
    \end{cases}
    \end{equation}
    \emph{has positive finite dimension (as a subspace of $W^{1,2}(\OO)$)}.
   \end{thm}
   By making use of Theorem \ref{thm.spectrumLL}, we can prove the
   Proposition~\ref{prop.spectrumGL} below.
   \begin{pro} \label{prop.spectrumGL}
    Let the assumptions 
    \emph{(H0)}-to-\emph{(H4)} be in force, and let $\GG_\LL$
    be the Green's operator for $\LL$
    \emph{(}thought of as an operator from 
    $C(\clOO,\R)$ into itself\emph{)}. 
    
    Then, the following facts hold true:
    \begin{itemize}
     \item[(i)] the spectral radius $r(\GG_\LL)$ of $\GG_\LL$ is strictly positive;
     \item[(ii)] there exists a non-negative $u_0\in C^{1,\alpha}(\clOO,\R)\setminus\{0\}$ such that
     $$\GG_\LL(u_0) = r(\GG_\LL)u_0$$
    \end{itemize}
   \end{pro}
   \begin{proof}
    (i)\,\,On account of Theorem \ref{thm.spectrumLL}, it is possible
    to find a real number $\sigma > 0$ and a 
    function $u_\sigma\in W^{1,2}(\OO)\setminus\{0\}$ such that
    $$\begin{cases}
    \LL u = \sigma u & \text{in $\OO$}, \\
    u\big|_{\de\OO} = 0,
    \end{cases}$$
    On the other hand, by applying 
    the classical Elliptic Regularity Theory to
    $\LL_\sigma := \LL - \sigma$
    (see, e.g., \cite[Corollary 8.35]{GT}), one can find
    a function $\hat{u}_\sigma\in C^{1,\alpha}(\clOO,\R)$
    such that 
    $$\text{$\hat{u}_\sigma \equiv u_\sigma$ a.e.\,on $\OO$};$$
    as a consequence, by the very definition of $\GG_\LL$ we infer that
    $$\GG_\LL(\hat{u}_\sigma) = \frac{1}{\sigma}\,\hat{u}_\sigma.$$
    This proves that $\lambda := 1/\sigma > 0$ lays in the (point-wise) spectrum of $\GG_\LL$
    (thought of as an operator from $C(\clOO,\R)$ into itself),
    and thus $r(\GG_\LL) > 0$. \medskip
    
    (ii)\,\,First of all, since $C^1(\clOO,\R)$ is continuously embedded
    in $C(\clOO,\R)$, 
    we straightforwardly derive from Proposition \ref{prop.propertiesG}-(ii) that
    $\GG_\LL$ is compact from $C(\clOO,\R)$ into itself; moreover,
    if we denote by $V_0$ the convex cone in $C(\clOO,\R)$ defined as
    $$V_0 := C(\clOO,\R^+) = \big\{u\in C(\clOO,\R):\,\text{$u\geq 0$ on $\OO$}\big\},$$
    we know from Proposition \ref{prop.propertiesG}-(iii) that $\GG_\LL(V_0)\subseteq V_0$.
    Since, obviously, $V_0-V_0$ is dense in $C(\clOO,\R)$ and since,
    by statement (i), the spectral radius $r(\GG_\LL)$ of $\GG_\LL$
    is strictly positive, we are entitled to apply
    Krein-Rutman's Theorem, ensuring that $r(\GG_\LL)$ is an eigenvalue
    of $\GG_\LL$ with positive eigenvector: this means that
    there exists $u_0\in V_0\setminus\{0\}$ such that
    $$\GG_\LL(u_0) = r(\GG_\LL)u_0\,\,\Longleftrightarrow\,\,
    u_0 = \frac{1}{r(\GG_\LL)}\,\GG_\LL(u_0).$$
    Now, since $u_0\in V_0\setminus\{0\}$, we have $u_0\geq 0$ and $u\not\equiv 0$ on $\OO$; moreover,
    reminding that $\GG_\LL$ maps $C(\clOO,\R)$ into $C^{1,\alpha}(\clOO,\R)$
    (see \eqref{eq.defoperatorG}), we derive that
    $u_0\in C^{1,\alpha}(\clOO,\R)$. 
    Gathering together all these facts,
    we conclude that $u_0\in C^{1,\alpha}(\clOO,\R)\setminus\{0\}$
    and that $u\geq 0$ on $\clOO$, as desired.
   \end{proof}

\section{Existence and non-existence results} \label{sec.existenceenon}

In this Section we study the solvability of the following system of second order 
elliptic differential equations subject to functional BCs
\begin{equation}
  \label{nellbvp-intro}
 \left\{
\begin{array}{lll}
 \LL_k u_k=\lambda_k\,f_k(x,u_1,\ldots,u_m, \nabla u_1,\ldots,\nabla u_m)
 &  \text{in $\OO$} & \qquad (k=1,2,\ldots,m), \\[0.15cm]
 u_k(x)=\eta_k\,\zeta_k(x)\,h_{k}[u_1,\ldots,u_m] & \text{for $x\in\de\OO$}
 & \qquad (k=1,2,\ldots,m),
\end{array}
\right.
\end{equation}
where, as in the Introduction, $m\geq 1$ is a fixed natural number, 
$\OO\subseteq\R^n$ is an open set and  
$\LL_1,\ldots,\LL_m$ are uniformly elliptic
PDOs on $\OO$ as in Section \ref{sec.preliminaries}. 
To be more precise, we suppose that 
\begin{itemize}
 \item[(I)] $\OO$ is bounded, connected and of class $C^{1,\alpha}$ for some
 $\alpha\in(0,1)$;
 \item[(II)] for every fixed $k = 1,\ldots, m$, the differential operator 
 $\LL_k$ satisfies assumptions
 (H1)-to-(H3) introduced in Section \ref{sec.preliminaries}, that is,
 \begin{itemize}
  \item[$(\ast)$]
   $\LL_k$ takes the divergence form \eqref{eq.generalformLL}, i.e.,
  \begin{equation*}
   \begin{split}
  \LL_k u & := 
   -\sum_{i,j = 1}^n
  \de_{x_i}\Big(a^{(k)}_{i,j}(x)\de_{x_j}u + b^{(k)}_i(x)u\Big)+
  \sum_{i = 1}^n c^{(k)}_i(x)\de_{x_i}u + d^{(k)}(x)u;
  \end{split}
  \end{equation*}
  \item[$(\ast)$] 
  the coefficient functions of $\LL_k$ belong to
  $C^{1,\alpha}(\clOO,\R)$;
  \item[$(\ast)$] the matrix $A^{(k)}(x) := \big(a_{i,j}^{(k)}(x)\big)_{i,j}$
  is symmetric for any $x\in\OO$;
  \item[$(\ast)$] $\LL_k$ is uniformly elliptic
  in $\OO$, i.e., there exists $\Lambda_k > 0$ such that
  $$\frac{1}{\Lambda_k}\|\xi\|^2\leq \sum_{i,j = 1}^na^{(k)}_{i,j}(x)\xi_i\xi_j\leq 
  \Lambda_k\|\xi\|^2\quad \text{for any $x\in\OO$ and $\xi\in\R^n\setminus\{0\}$};$$
  \item[$(\ast)$] for every non-negative function $\varphi\in C_0^\infty(\OO,\R)$ one has
  $$\int_{\mathcal{O}}
  \Big(d^{(k)}\varphi+{\textstyle\sum_{i = 1}^n}b^{(k)}_i\de_{x_i}\varphi\Big)\,\d x,\quad
  \int_{\mathcal{O}}
  \Big(d^{(k)}\varphi+{\textstyle\sum_{i = 1}^n}c_i^{(k)}\de_{x_i}\varphi\Big)\,\d x\geq 0.$$
 \end{itemize}
\end{itemize}
Furthermore, for every fixed $k = 1,\ldots, m$ we also assume that
\begin{itemize}
  \item[(III)] $f_k$ is a real-valued function defined on $\clOO\times\R^m\times\R^{nm}$;
  \item[(IV)] $h_k$ is a real-valued operator defined on the space $C^1(\clOO,\R^m)$;
  \item[(V)] $\zeta_k\in C^{1,\alpha}(\clOO,\R)$ and $\zeta_k\geq 0$ on $\OO$;
  \item[(VI)]$ \lambda_k,\,\eta_k$ are non-negative real parameters. 
\end{itemize}
 Throughout the sequel, if $u_1,\ldots,u_m$ are real-valued functions
 defined on $\OO$, we set
 $$\bldu(x) := \big(u_1(x),\ldots,u_m(x)\big) \qquad(x\in\OO).$$
 If, in addition, 
 $\bldu\in C^1(\OO,\R^m)$ (that is, $u_1,\ldots,u_m\in C^1(\OO,\R)$), we define
 $$D\bldu(x) := \big(\nabla u_1(x),\ldots,\nabla u_m(x)\big) \qquad (x\in \OO).$$
 Now, in view of assumptions (I)-(II), all the results presented
 in Section \ref{sec.preliminaries} can be applied to each operator
 $\LL_k$ (for a fixed $k = 1,\ldots,m$); in particular, for every
 $\mathfrak{f}\in  C(\clOO,\R)$ there exists a unique
 solution $u_{\mathfrak{f}}\in C^{1,\alpha}(\clOO,\R)$ of the Poisson problem
 \begin{equation}
  \label{eqelliptic}
	\begin{cases}
	\LL_k u = \mathfrak{f} & \text{in $\OO$}, \\
	u\big|_{\de\OO} = 0.
	\end{cases}
\end{equation}
 Furthermore, since the function $\zeta_k$ belongs to 
 $C^{1,\alpha}(\clOO,\R)$ (see
 assumption (V)), there exists a unique so\-lu\-tion 
 $\gamma_k \in C^{1,\alpha}(\clOO,\R)$ of the Dirichlet problem
 \begin{equation}
  \label{eqellipticbc}
   \begin{cases}
	\LL_k u = 0 & \text{in $\OO$}, \\
	u\big|_{\de\OO} = \zeta_k.
	\end{cases}
\end{equation}
 We then denote by $\GG_k$ the Green's operator $\GG_{\LL_k}$ for $\LL_k$
 defined in \eqref{eq.defoperatorG}, and we indicate
 by $g_k$ the Green's function $g_{\LL_k}$ for the operator
 $\GG_k$ defined through Theorem \ref{thm.mainexistenceGreen}.
 We remind that, if $\mathfrak{f}\in C(\clOO,\R)$ is arbitrary fixed,
 $\GG_k(\mathfrak{f})$ is the unique solution in $C^{1,\alpha}(\clOO,\R)$ of
 the Poisson problem \eqref{eqelliptic}; moreover, we have the representation
 formulas
 \begin{align*}
  \GG_k(\mathfrak{f})(x) = \int_{\OO}g_k(y;x)\mathfrak{f}(y)\,\d y \qquad\text{and}\qquad
  \de_{x_i}\GG_k(\mathfrak{f})(x)
  = \int_{\OO}\de_{x_i}g_k(y;x)\mathfrak{f}(y)\,\d y,
 \end{align*}
 holding true for a.e.\,$x\in\OO$ and any $i = 1,\ldots,n$
 (see Theorem \ref{thm.mainexistenceGreen} and
 Proposition \ref{prop.Piweakderivative}). 
 
 Finally, according to Proposition \ref{prop.spectrumGL}, 
 we denote by $r_k = r(\GG_k) > 0$ the spectral radius
 of the operator $\GG_k$ (thought of as an operator from $C^1(\clOO,\R)$ into itself)
 and
 we fix once and for all a function $\varphi_k\in C^{1,\alpha}(\clOO,\R)\setminus\{0\}$ such that
 (setting $\mu_k := 1/r_k$)
 \begin{equation} \label{eq.defivarphik}
 \varphi_k = \mu_k\,\GG_k(\varphi_k)\qquad\text{and}\qquad\text{$\varphi_k\geq 0$ on $\OO$}.
 \end{equation}
 Now that we have properly introduced all the \textquotedblleft mathematical objects''
 appearing in the problem~\eqref{nellbvp-intro}, it is opportune
 to define what we mean by a \emph{solution} of this problem.
 
 To this end, we first fix some notation. 
 For every index $k\in\{1,\ldots,m\}$, we denote by $\mathcal{F}_k$
 the so-called \emph{superposition (Nemytskii)} operator
 associated with $f_k$, that is,
 $$\mathcal{F}_k:C^1(\clOO,\R^m)\to C(\clOO,\R),\qquad
 \mathcal{F}_k(\bldu) := f_{{k}}(x,\bldu, D\bldu).$$
 Moreover, we consider the operators $\mathcal{T}, \Gamma:
 C^1(\clOO,\R^m)\to C^{1}(\clOO,\R^m)$ defined by
 $$\mathcal{T}(\bldu) = \big(\lambda_k\,(\GG_k\circ \mathcal{F}_k)(\bldu)\big)_{k = 1,\ldots,m}
 \qquad\text{and}\qquad
  \Gamma(\bldu) := \big(\eta_k\,\gamma_k(x)\,h_k[\bldu]\big)_{k = 1,\ldots,m}.$$
 We can now give the definition of \emph{solution} of the problem \eqref{nellbvp-intro}.
 \begin{defn} \label{def.solBVPmain}
  We say that a function  $\bldu\in C^1(\clOO, \R^m)$ 
  is a {\it weak solution} of the system~\eqref{nellbvp-intro} 
  if $\bldu$ is a fixed point of the operator $\mathcal{T}+\Gamma$, that is, 
$$
\bldu=\mathcal{T}(\bldu)+
\Gamma(\bldu)=
\big(\lambda_k\,(\GG_k\circ \mathcal{F}_k)(\bldu) + 
\eta_k\,\gamma_k(x)\,h_k[\bldu]\big)_{k = 1,\ldots,m}.
$$
If, in addition, 
the components of $\bldu$ are non-negative 
and $u_j\not\equiv 0$ for some $j$, we say that $\bldu$ is a 
\emph{nonzero positive solution} of the system~\eqref{nellbvp-intro}.
\end{defn}
For our existence result, we make use of the following proposition that states the main properties of the classical fixed point index, for more details see~\cite{Amann-rev, guolak}. In what follows the closure and the boundary of subsets of a cone $\hat{P}$ are understood to be relative to~$\hat{P}$.

\begin{pro}\label{propindex}
Let $X$ be a real Banach space and let $\hat{P}\subset X$ be a cone. Let $D$ be an open bounded set of $X$ with $0\in D\cap \hat{P}$ and
$\overline{D\cap \hat{P}}\ne \hat{P}$. 
Assume that $ T:\overline{D\cap \hat{P}}\to \hat{P}$ is a compact operator such that
$x\neq T(x)$ for $x\in \partial (D\cap \hat{P})$. \medskip

Then the fixed point index
 $i_{\hat{P}}( T, D\cap \hat{P})$ has the following properties:
 
\begin{itemize}

\item[\emph{(i)}] If there exists $e\in \hat{P}\setminus \{0\}$
such that $x\neq T(x)+\sigma e$ for all $x\in \partial (D\cap \hat{P})$ and all
$\sigma>0$, then $i_{\hat{P}}( T, D\cap \hat{P})=0$.

\item[\emph{(ii)}] If $ T(x) \neq \sigma x$ for all $x\in
\partial  (D\cap \hat{P})$ and all $\sigma > 1$, then $i_{\hat{P}}( T, D\cap \hat{P})=1$.

\item[\emph{(iii)}] Let $D^{1}$ be open bounded in $X$ such that
$(\overline{D^{1}\cap \hat{P}})\subset (D\cap \hat{P})$. If 
$i_{\hat{P}}(T, D\cap \hat{P})=1$ and 
$i_{\hat{P}}(T, D^{1}\cap \hat{P})=0$, then $T$ has a fixed point in $(D\cap \hat{P})\setminus
(\overline{D^{1}\cap \hat{P}})$. The same holds if 
$i_{\hat{P}}(T, D\cap \hat{P})=0$ and $i_{\hat{P}}(T, D^{1}\cap \hat{P})=1$.
\end{itemize}
\end{pro}
We can now state a result regarding the existence of positive solutions 
for the system~\eqref{nellbvp-intro}.

In the sequel, we will consider on the space $\R^s$
(where $s$ will be either $m,n$ or $mn$)
the following maximum norm
\begin{equation} \label{eq.norminRs}
 |\mathbf{v}| := \max_{i = 1,\ldots,s}|v_i|\qquad (\text{if $\mathbf{v} = (v_1,\ldots,v_s)$}).
 \end{equation}

We will work in the Banach space $C(\clOO,\R^m)$ endowed with the norm
$$\|\mathbf{z}\|_{\infty}=\displaystyle\max_{x\in \clOO}
|\bldz(x)|
:= \max\big\{\|z_1\|_\infty,\ldots,\|z_m\|_\infty\big\}$$
where $\mathbf{z}
= (z_1,\ldots,z_m)\in C(\clOO,\R^m)$, 
compare also with \eqref{eq.defnormainfty}.
Moreover, we will consider
the Banach space $C^1(\clOO,\R^m)$ endowed with the norm
\begin{equation} \label{eq.topologyCRm}
\begin{split}
\|\bldu\|_{C^1(\clOO,\R^m)}
& := \max\Big\{
 \max_{k=1,2,\ldots,m}\|u_k\|_{\infty}, 
 \max_{k=1,2,\ldots,m}\|\nabla u_k\|_{\infty}\Big\} \\[0.05cm]
 & = 
\max\Big\{\|u_k\|_\infty,\,\|\de_{x_l}u_k\|_\infty:\,
 \text{$k = 1,\ldots,m$ and $l = 1,\ldots,n$}\Big\};
 \end{split}
\end{equation}
notice that \eqref{eq.topologyCRm} reduces to
\eqref{eq.defnormaC1scalar} when $m = 1$.
Given a finite sequence
$\varrho = \{\rho_k\}_{k=1}^m\subseteq (0,+\infty)$, we 
define
\begin{equation} \label{eq.defiIBvarrho}
I(\varrho)=\prod_{k = 1}^m\,[0,\rho_k] \quad \mbox{ and } \quad
R(\varrho)=\prod_{k = 1}^m {R}_{\rho_k}
\end{equation}
where $R_\rho=\{\mathbf{v}\in \R^{n}: |\mathbf{v}| \leq \rho\}$ (for $t > 0$); we also introduce, with abuse of notation, the sets
\begin{equation} \label{eq.deficonPPvarrho}
\begin{split}
& P := \Big\{\bldu\in C^1(\clOO,\R^m):\,\text{$u_k\geq 0$ on $\OO$ for every $k = 1,\ldots,m$}\Big\}
\qquad\text{and} \\[0.15cm]
& {P}(\varrho)=\Big\{\bldu\in C^1(\clOO,\R^m): \bldu(x) \in I(\varrho) \mbox{ and } D\bldu(x) 
\in R(\varrho) \mbox{ for all  } x \in\clOO \Big\}\subseteq P.
\end{split}
\end{equation}
\begin{thm}\label{thmsol} 
Let the assumptions \emph{(I)}-to-\emph{(VI)} be in force. Moreover, 
let us suppose that one can find
a finite sequence
$\varrho = \{\rho_k\}_{k = 1}^m\subseteq(0,\infty)$ satisfying the following hypotheses:
\begin{itemize}
\item[\emph{(a)}] For every $k = 1,\ldots,m$, one has that
\begin{itemize}
 \item[$\mathrm{(a)}_1$] $f_k$ continuous and non-negative on 
 $\clOO\times I(\varrho)\times R(\varrho)$;
 \item[$\mathrm{(a)}_2$] $h_k$ continuous, non-negative and bounded on $P(\varrho)$.
\end{itemize}

\item[\emph{(b)}] There exist $\delta \in (0,+\infty)$, $k_0\in \{1,2,\ldots,{m}\}$ 
and $\rho_0 \in (0,\displaystyle\min_{k=1,\ldots,m}{\rho_k})$ such that
\begin{equation} \label{eq.mainestimk0}
 f_{k_0}(x,\bldz,\bldw)\ge \delta z_{{k_0}}\quad \text{for every $(x,\bldz,\bldw)\in 
\clOO\times I_{0}\times B_0$},
\end{equation}
where $I_{0}:=\prod_{i=1}^m [0,\rho_0]$ and
$R_0:=\prod_{k=1}^m {R}_{\rho_0}$.

\item[\emph{(c)}] Setting, for every $k = 1,\ldots,m$,
\begin{equation} \label{eq.defiMkHk}
\begin{split}
 & M_{k}:=\max \Big\{f_{{k}}(x,\bldz,\bldw): (x,\bldz,\bldw)\in \clOO\times I(\varrho)\times 
 R(\varrho)\Big\}\qquad \text{and}
\\
& H_k:=\sup_{\bldu\in {P}(\varrho)}h_{k}[\bldu].
\end{split}
\end{equation}
the following inequalities are satisfied: 
\begin{itemize}
 \item[$\mathrm{(c)}_1$] ${\mu_{k_{0}}}\leq \delta\lambda_{k_0}$; \vspace*{0.08cm}
 
 \item[$\mathrm{(c)}_2$] 
 $\lambda_k\,M_k\,\|\GG_k(\hat{1})\|_{\infty} + 
 \eta_k\,H_k\|\gamma_k\|_{\infty} \leq \rho_k$; \vspace*{0.08cm}
 
 \item[$\mathrm{(c)}_3$]
 for any $l = 1,\ldots,n$ we have $\lambda_k\,M_k\,G_{k,l} 
 + \eta_k\,H_k\|\de_{x_l}\gamma_k\|_{\infty} \leq \rho_k$,
 where
\begin{equation}\label{intgreen}
G_{k,l}:= \sup_{x \in \OO} 
\int_{\OO}\left|\de_{x_l}g_k(y;x)\right|\,\d y\qquad
(\text{see Lemma \ref{lem.propertiesGLintegral}}).
\end{equation}
\end{itemize}
\end{itemize}
Then the system~\eqref{nellbvp-intro} has a non-zero positive weak solution $\bldu\in C^{1}
(\clOO,\R^m)$ such that 
\begin{equation} \label{eq.estimnonzero}
\|\bldu\|_{C^1(\clOO,\R^m)}\geq \rho_0\qquad
\text{and} \qquad \text{$\|u_k\|_{\infty}\leq \rho_k$ for every $k=1,\ldots,m$}.
\end{equation}
\end{thm}
\begin{proof}
For the sake of readability, we split the proof into different steps. \medskip

\textsc{Step I:} We first prove that the operator
$\AA := \mathcal{T}+\Gamma$ maps $P(\varrho)$ into $P$.

To this end, let $\bldu\in P(\varrho)$ and let
$k\in\{1,\ldots,m\}$ be fixed. Since $\bldu\in P(\varrho)$, from as\-sump\-tion 
$(\mathrm{a})_2$ we derive that
$h_k[\bldu]\geq 0$; moreover, since
$\gamma_k\geq 0$ on $\clOO$ (see Proposition \ref{prop.propertiesG}-(iii))
and since, by assumption (VI), $\eta_k\geq 0$, we get
\begin{equation} \label{eq.Gammakpositive}
 \Gamma_k(\bldu)(x) = \eta_k\,\gamma_k(x)\,h_k[\bldu]\geq 0\quad \text{for all $x\in\clOO$}.
\end{equation}
On the other hand, since $\bldu\in P(\varrho)$, by assumption $\mathrm{(a)_1}$ we
also have that
$$\mathcal{F}_k(\bldu)(x) = f_k(x,\bldu(x), D\bldu(x)) \geq 0\quad\text{for all
$x\in\clOO$};$$
as a consequence, from Proposition \ref{prop.propertiesG}-(iii) we derive that
$\GG_k(\mathcal{F}_k(\bldu))\geq 0$ on $\clOO$. Finally,
since $\lambda_k\geq 0$ (by assumption (IV)), we get
\begin{equation} \label{eq.Tkpositive}
 \mathcal{T}_k(\bldu)(x) = \lambda_k\,\GG_k(\mathcal{F}_k(\bldu)(x))\geq 0\quad
 \text{for every $x\in\clOO$}.
\end{equation}
By \eqref{eq.Gammakpositive}, \eqref{eq.Tkpositive} and the arbitrariness
of $k$, we conclude that
$\AA(P(\varrho))\subseteq P$. \medskip

\textsc{Step II:} We now prove that $\AA:P(\varrho)\to P$ is compact.
To this end, let $\{\bldu_j\}_{j\in\N}$ be a bounded sequence in
$P(\varrho)$, and let $k\in\{1,\ldots,m\}$ be fixed. Since
$h_k$ is non-negative and bounded on $P(\varrho)$ (see assumption $\mathrm{(a)_2}$), 
the sequence $\{h_k[\bldu_j]\}_j$ is bounded in $(0,\infty)$; as a consequence,
there exists $\theta_0\in[0,\infty)$ such that
(up to a sub-sequence)
\begin{equation} \label{eq.convergenceGammaut}
 \lim_{j\to\infty}\Gamma_k(\bldu_j) = \eta_k\,\gamma_k(x)\,\theta_0\qquad\text{in
 $C^1(\clOO,\R)$}.
\end{equation}
On the other hand, since $\{\bldu_j\}_j\subseteq P(\varrho)$
and since $f_k$ is continuous on $\clOO\times I(\varrho)\times R(\varrho)$
(see as\-sump\-tion $\mathrm{(a)_1}$), we have (using the notation in \eqref{eq.defiMkHk})
$$\|\mathcal{F}(\bldu_j)\|_{\infty}\leq M_k \qquad\text{for every $j\in\N$}.$$
As a consequence, since the operator $\GG_k$ is compact (as an operator
from $C(\clOO,\R)$ into $C^1(\clOO,\R)$, see Proposition \ref{prop.propertiesG}-(ii)), 
it is possible to find a function $w_k\in C^1(\clOO,\R)$ such that
(again by possibly passing to a sub-sequence)
\begin{equation} \label{eq.convergenceTut}
 \lim_{j\to\infty}\mathcal{T}_k(\bldu_j)
 = \lim_{j\to\infty}\big(\lambda_k\,\GG_k(\mathcal{F}_k(\bldu_j))\big) 
 = \lambda_k\,w_k \qquad\text{in $C^1(\clOO,\R)$}.
\end{equation}
 Gathering together \eqref{eq.convergenceGammaut}, \eqref{eq.convergenceTut}
 and \eqref{eq.topologyCRm},
 we infer that (up to a suitable sub-sequence)
 $$\lim_{j\to\infty}\AA(\bldu_j) = \big(\lambda_k\,w_k + 
 \eta_k\,\gamma_k\,\theta_0\big)_{k = 1,\ldots,m} =: \widetilde{\mathbf{u}} 
 \qquad\text{in $C^1(\clOO,\R^m)$}.$$
 Finally, since $\{\AA(\bldu_j)\}_j\subseteq P$ (by Step I) and since
 $P$ is closed, we conclude that $\widetilde{\mathbf{u}}\in P$; this proves
 the compactness of $\AA$ (as an operator from $P(\varrho)$ to $P$). \medskip
 
 \noindent To proceed further, we consider the set $P_0\subseteq C^1(\clOO,\R^m)$
 defined as follows:
 $$P_0=\Big\{\bldu\in C^1(\clOO,\R^m):\, 
 \bldu(x) \in I_0 \mbox{ and } D\bldu(x) \in R_0 \mbox{ for all  } x \in\clOO \Big\}\subseteq 
 P(\varrho),$$ 
 where $I_0$ and $B_0$ are as in assumption (b). Now, if
 the operator $\AA = \mathcal{T}+\Gamma$ has a fixed point  
 $\bldu_0\in \de P_0\cup\de P(\varrho)$ (where the boundaries
 are both relative to $P$), then $\bldu_0$ is a solution
 of problem
 \eqref{nellbvp-intro} satisfying \eqref{eq.estimnonzero}, and the theorem is proved.
 
 If, instead, $\AA$
 is fixed-point free on $\de P_0\cup\de P(\varrho)$, the fixed-point indexes
 $$i_{{P}}(\AA, \mathrm{int}(P_0)\cap P) \qquad\text{and}
 \qquad i_{{P}}(\AA, \mathrm{int}(P(\varrho))\cap P)$$
 are well-defined. Assuming this last possibility, we consider the following steps. \medskip
 
 \textsc{Step III:} In this step we prove the following fact: 
 \begin{equation} \label{eq.toproveIndexone}
  i_{{P}}(\AA, \mathrm{int}(P(\varrho))\cap P) = 1.
  \end{equation}
 According to Proposition \ref{propindex}-(ii), to prove
 \eqref{eq.toproveIndexone} it suffices to show that
 \begin{equation} \label{eq.toprovenofixedpointBDone}
  \AA(\bldu) \neq \sigma\,\bldu \qquad\text{for every $\bldu\in\de P(\varrho)$
  and every $\sigma > 1$},
 \end{equation}
 To
 establish \eqref{eq.toprovenofixedpointBDone} we argue by contradiction,
 and we suppose that there exist 
 a function $\bldu\in \de P(\varrho)$ and a real $\sigma >1$ such that 
 $$\sigma\bldu= \AA(\bldu) = \mathcal{T}(\bldu)+\Gamma(\bldu).$$ 
 Since $\bldu\in \de P(\varrho)$, 
there exists an index $k\in \{1,\ldots,m\}$ such 
that either 
$$\text{$\| u_k\|_{\infty} = \rho_k$ or $\| \nabla u_k\|_{\infty} = \rho_k$}.$$
We then distinguish these two cases.
\begin{itemize}
 \item $\| u_k \|_{\infty} = \rho_k$. In this case, by
 exploiting assumption $\mathrm{(a)_1}$ and
 \eqref{eq.defiMkHk}, we have
 \begin{equation} \label{eq.touseWMP}
  0\leq \mathcal{F}_{{k}}(\bldu)(x) = 
  f_k(x,\bldu(x),D\bldu(x)) \leq M_k \qquad\text{for all $x\in\clOO$};
 \end{equation}
 from this, we derive the following chain of inequalities:
\begin{equation} \label{idx1in}
\begin{split}
 \sigma u_k (x) & =
 \lambda_k\,\GG_k\big(\mathcal{F}_k(\bldu)\big)(x) +
  \eta_k\,\gamma_k(x)\,h_k[\bldu] \\
 & \big(\text{since $\GG_k\big(M_k\hat{1}-\mathcal{F}_k(\bldu)\big)\geq 0$,
 see \eqref{eq.touseWMP}
 and Proposition \ref{prop.propertiesG}-(iii)}\big) \\
  & \leq \lambda_k\,\GG_k\big(M_k\hat{1}\big)(x) +
  \eta_k\,\gamma_k(x)\,h_k[\bldu] \\
  & \big(\text{since $\bldu\in \de P(\varrho)\subseteq P(\varrho)$, 
  see \eqref{eq.defiMkHk}}\big) \\
&  \leq \big\|\lambda_k\,\GG_k\big(M_k\hat{1}\big)\big\|_{\infty}
+ \big\|\eta_k\,H_k\,\gamma_k\big\|_{\infty}  \\
& 
= \lambda_k\,M_k\,\big\|\GG_k(\hat{1})\|_{\infty} 
+ \eta_k\,H_k\,\|\gamma_k\|_{\infty} 
 \leq \rho_k \qquad (\text{see assumption $\mathrm{(c)}_2$}).
\end{split}
\end{equation}
 As a consequence, by taking the supremum for $x\in \clOO$ in \eqref{idx1in}
(and by reminding that $\bldu\in \de P(\varrho)\subseteq P(\varrho)$), we then  obtain 
$$\sup_{x\in\clOO}|\sigma\,u_k(x)| \leq \sigma\,\rho_k \leq \rho_k,$$ 
which is clearly a contradiction (since $\sigma > 1$). \medskip

\item $\|\nabla u_k\|_{\infty} = \rho_k$. In this case, 
 by the very definition of $\|\cdot\|_\infty$, there exists
  $l\in \{1,\ldots,n\}$ such that 
 $\|\de_{x_l}u_k\|_{\infty}=\rho_k.$
Moreover, by Proposition \ref{prop.Piweakderivative} we have
$$
\sigma\,\de_{x_l}u_k(x) =
\lambda_k\,\int_{\OO}\de_{x_l}g_k(y;x)f_k(x,\bldu(y), D\bldu(y))\,\d y 
+ \eta_k\,\de_{x_l}\gamma_k(x)\,h_k[\bldu],
$$
for a.e.\,$x\in\OO$. By means of this representation formula, we then obtain
\begin{equation*}
 \begin{split}
 & \sigma\,\big|\de_{x_l}u_k(x)| \\ 
& \qquad \leq
\lambda_k\,\int_{\OO}\big|\de_{x_l}g_k(y;x)f_k(x,\bldu(y), D\bldu(y))\big|\,\d y +
 \eta_k\,h_k\,[\bldu]\,\big|\de_{x_l}\gamma_k(x)\big| \\
 & \qquad 
 \big(\text{since $\bldu\in \de P(\varrho)\subseteq P(\varrho)$, 
 see also \eqref{eq.touseWMP}}\big) \\
 & \qquad \leq 
 \lambda_k\,M_k\,\int_{\OO}|\de_{x_l}g_k(y;x)|\,\d y
 +
 \eta_k\,H_k\,|\de_{x_l}\gamma_k(x)| \\
 & \qquad
 \leq \lambda_k\,M_k\,G_{k,l} 
 + \eta_k\,H_k\,\|\de_{x_l}\gamma_k\|_{\infty}
 \qquad(\text{see \eqref{intgreen}}).
 \end{split}
\end{equation*}
As a consequence, by taking the supremum for $x\in \clOO$ in \eqref{idx1in}
(and by reminding that $\|\de_{x_l}u_k\|_\infty = \rho_k$), from
assumption $\mathrm{(c)}_3$ we infer that 
$$
\sup_{x\in\clOO}
\big(\sigma\,\big|\de_{x_l}u_k(x)|\big)
= \sigma\,\rho_k
\leq \lambda_k\,M_k\,G_{k,l} 
 + \eta_k\,H_k\,\|\de_{x_l}\gamma_k\|_{\infty}
\leq \rho_k
$$
which is clearly a contradiction (as $\sigma > 1$).
\end{itemize}
This completes the demonstration of \eqref{eq.toprovenofixedpointBDone}. \medskip

\textsc{Step IV:} In this last step we prove the following fact:
 \begin{equation} \label{eq.toproveIndexzero}
  i_{{P}}(\AA, \mathrm{int}(P_0)\cap P) = 0.
  \end{equation}
 According to Proposition \ref{propindex}-(i), to prove
 \eqref{eq.toproveIndexzero} it suffices to show that
 there exists a suitable function $e\in P\setminus\{0\}$ satisfying the property
 \begin{equation} \label{eq.toprovenofixedpointBDzero}
  \AA(\bldu) + \sigma e \neq \bldu \qquad\text{for every $\bldu\in\de P_0$
  and every $\sigma > 0$}.
 \end{equation}
 To
 establish \eqref{eq.toprovenofixedpointBDzero}, we let
 $e := (\varphi_1,\ldots,\varphi_m)$ (where $\varphi_1,\ldots,\varphi_m$
 are as in \eqref{eq.defivarphik}) and we argue by contradiction: we thus
 suppose that there exist 
 $\bldu\in \partial P_0$ and $\sigma  >0$ such that
$$
\bldu= \AA(\bldu) + \sigma e = \mathcal{T}(\bldu)+\Gamma(\bldu)+\sigma e. 
$$
Since $\bldu\in \de P_0\subseteq P_0\subseteq P(\varrho)$
(by definition of $P_0$, see assumption (b)), we know from Step I
that $\AA(\bldu)\in P$; as a consequence, if $k_0$ is as in assumption (b), we have
$$u_{k_0} = \AA(\bldu)_{k_0} 
+ \sigma\,\varphi_{k_0}\geq \sigma  \varphi_{k_0} \qquad\text{on $\clOO$}.$$
Furthermore, by exploiting once again assumption (b) we get
\begin{equation} \label{eq.estimFk0varphi}
 \mathcal{F}_{k_0}(\bldu) = f_{k_0}(x,\bldu(x),D\bldu(x))
  \geq \delta u_{k_0}(x)\geq \delta\sigma\varphi_{k_0}(x) \qquad\text{for all$x\in\clOO$}. 
\end{equation}
Gathering together all these facts, for every $x\in \clOO$ we have
\begin{equation*}
\begin{split}
 u_{k_0}(x) & =
  \lambda_{k_0} \GG_{k_0}\big(\mathcal{F}_{k_0}(\bldu)\big)(x)
  +\eta_{k_0}\,\gamma_{k_0} (x)\,h_{k_0}[\bldu]  
  + \sigma\,\varphi_{k_0} (x) \\
  & \big(\text{since $\GG_{k_0}\big(\mathcal{F}_{k_0}(\bldu)
  - \delta\sigma\varphi_{k_0}\big)\geq 0$,
 see \eqref{eq.estimFk0varphi} and Proposition \ref{prop.propertiesG}-(iii)}\big) \\
 & \geq 
\lambda_{k_0}\,\GG_{k_0}(\delta\sigma\varphi_{k_0})(x) +
\sigma\varphi_{k_0} (x) \\
& \big(\text{since $\varphi_{k_0}$ is an eigenfunction of $\GG_{k_0}$, see
\eqref{eq.defivarphik}}\big) \\
& =  \frac{\delta\lambda_{k_0}}{\mu_{k_0}}\cdot\sigma\varphi_{k_0}(x)
 + \sigma\varphi_{k_0} (x)  \geq 2\sigma  \varphi_{k_0} (x)
 \qquad (\text{see assumption $\mathrm{(c)}_1$}).
\end{split}
\end{equation*}
By iterating the above argument, for every $x\in\clOO$ we get
$$
u_{k_0} (x)\geq p\sigma\varphi_{k_0}(x) \qquad \text{for every $p\in\mathbb{N}$},
$$
but this is contradiction with the boundedness of $u_{k_0}\in C^1(\clOO,\R)$
(as $\varphi_{k_0}\not\equiv 0$). \medskip

\noindent We are now ready to conclude the proof of the theorem: in fact,
by combining \eqref{eq.toproveIndexone}, \eqref{eq.toproveIndexzero} and Proposition
\ref{propindex}-(iii), we infer the existence of a fixed point 
$$\bldu_0\in \big(\mathrm{int}(P(\varrho))\cap P\big) \setminus P_0$$ 
of $\AA = \mathcal{T}+\Gamma$; thus, $\bldu_0$ is a solution of
\eqref{nellbvp-intro} satisfying \eqref{eq.estimnonzero}.
\end{proof}
\begin{rem} \label{rem.regulC1alfasol}
 Let the assumption and the notation of Theorem \ref{thmsol}
 do apply. We have already pointed out that, since $\zeta_1,\ldots,\zeta_m\in
 C^{1,\alpha}(\clOO,\R)$ (see assumption (V)), one has
 $$\gamma_k\in C^{1,\alpha}(\clOO,\R) \qquad\text{for every $k = 1,\ldots,m$}.$$
 As a consequence, the operator $\Gamma$ maps $C^1(\clOO,\R^m)$ into $C^{1,\alpha}(\clOO,\R^m)$.
 On the other hand, since the operators $\GG_1,\ldots,\GG_m$ map
 $C(\clOO,\R)$ into $C^{1,\alpha}(\clOO,\R)$, we also have that
 $$\mathcal{T}(C^1(\clOO,\R^m))\subseteq C^{1,\alpha}(\clOO,\R^m).$$
 Gathering together all these facts, we conclude that \emph{any} weak solution
 of \eqref{nellbvp-intro} (i.e., any fixed point
 of $\AA = \mathcal{T}+\Gamma$ in $C^1(\clOO,\R^m)$) 
 is actually of class $C^{1,\alpha}$ on $\clOO$.
\end{rem}

An elementary argument yields the following non-existence result.

\begin{thm}\label{thmnonex} 
Let the assumptions \emph{(I)}-to-\emph{(IV)} be in force. Moreover, 
let us suppose that there exists
a finite sequence
$\varrho = \{\rho_k\}_{k = 1}^m\subseteq(0,\infty)$
such that, for every $k = 1,\ldots,m$, the following conditions hold:

\begin{itemize}
\item[\emph{(a)}] $f_k$ is continuous on 
$\clOO\times I(\varrho)\times R(\varrho)$, and there exist $\tau_k \in (0,+\infty)$
such that
$$0\le f_{k}(x,\bldz,\bldw)\le \tau_k z_k\quad \text{for every $(x,\bldz,\bldw)\in 
\clOO\times I(\varrho)\times R(\varrho)$},$$ 

\item[\emph{(b)}] $h_k$ is continuous on $P(\varrho)$ and there exist $\xi_k \in (0,+\infty)$
such that
$$ h_{k}[\bldu] \le \xi_k \cdot \|\bldu\|_\infty, \quad \text{for every $u\in 
P(\varrho)$},$$

\item[\emph{(c)}] the following inequality holds: 
\begin{equation} \label{eq.estimleq1absurd}
 \lambda_k\tau_k\,\|\GG_k(\hat{1})\|_{\infty} + 
 \eta_k\,\xi_k\|\gamma_k\|_{\infty} <1.
 \end{equation}
\end{itemize}
Then the system~\eqref{nellbvp-intro} has at most the zero solution in $P(\varrho)$.
\end{thm}

\begin{proof}
 We argue by contradiction and we assume that \eqref{nellbvp-intro}
 has a solution $\bldu\in P(\varrho)\setminus\{0\}$. According to Definition
 \ref{def.solBVPmain}, this means that $\bldu$ is a fixed point of 
 the operator
 $\mathcal{A} = \mathcal{T}+\Gamma$.
 Setting $\rho := \|\bldu\|_\infty > 0$, we let 
 $j\in \{1,2,\ldots,m\}$ such that  
 \begin{equation} \label{eq.choicejabsurd}
 \|u_{j}\|_{\infty}=\rho.
 \end{equation}
 For every $x\in\clOO$, we then have 
 \begin{equation} \label{eq.touseabsurd}
  0\leq \mathcal{F}_{j}(\bldu)(x) = 
  f_j(x,\bldu(x),D\bldu(x)) \leq \tau_j u_j(x) \leq \tau_j \rho;
 \end{equation}
 from this, we obtain
\begin{equation} \label{nonext}
\begin{split}
u_{j}(x) & =
 \lambda_j\,\GG_j\big(\mathcal{F}_j(\bldu)\big)(x) +
  \eta_j\,\gamma_j(x)\,h_j[\bldu] \\ 
  & (\text{since $\GG_j(\tau_j\rho\cdot\hat{1}-\mathcal{F}_j(\bldu))\geq 0$, see 
  \eqref{eq.touseabsurd} and 
  Proposition \ref{prop.propertiesG}-(iii)}) \\
  & \leq \lambda_j\,\GG_j\big(\tau_j \rho \hat{1}\big)(x) +
  \eta_j\,\gamma_j(x)\,h_j[\bldu] \\
  & (\text{by assumption (b) and since $\|\bldu\|_\infty = \rho$}) \\
  & \leq \big\|\lambda_j\,\GG_j\big(\tau_j \rho\hat{1}\big)\big\|_{\infty}
+ \big\|\eta_j\,\xi_j \rho\,\gamma_j\big\|_{\infty}  \\
& = \left(\lambda_j\,\tau_j\,\big\|\GG_k(\hat{1})\|_{\infty} 
+ \eta_j\,\xi_j\,\|\gamma_j\|_{\infty}\right) \rho.
\end{split}
\end{equation}
By taking the supremum in~\eqref{nonext} for $x\in \clOO$, 
from \eqref{eq.estimleq1absurd} and \eqref{eq.choicejabsurd} 
we finally obtain 
$$\rho = \sup_{x\in\clOO}u_j(x) \leq 
\left(\lambda_j\,\tau_j\,\big\|\GG_k(\hat{1})\|_{\infty} 
+ \eta_j\,\xi_j\,\|\gamma_j\|_{\infty}\right)\rho < \rho,$$
a contradiction. Thus, problem \eqref{nellbvp-intro}
cannot have nonzero solutions  in $P(\varrho)$.
\end{proof}

\section{Examples} \label{sec.examples}
  In this last section we present a couple of concrete examples
  illustrating the applicability of 
  our main results, namely Theorems \ref{thmsol}
  and \ref{thmnonex}.

\begin{ex} \label{exm.existence}
 On Euclidean space $\R^3$, let us consider the following BVP
 \begin{equation} \label{eq.BVPexistence}
  \begin{cases}
   -\Delta u_1 = \lambda_1\,e^{u_1}\big(1+|\nabla u_2|^2\big) & \text{in $B$}, \\
   -\Delta u_2 = \lambda_2\,(16-u_2^2)\cos\big(\langle \nabla u_1,\nabla u_2\rangle\big)
   & \text{in $B$}, \\
   u_1\big|_{\de B} = \eta_1\big(u_1(0)+u_2(0)\big), \\
   u_2\big|_{\de B} = \eta_2\int_{\de B_1}
   u_1(1-|\nabla u_2|^2)\,\d\sigma,
  \end{cases}
 \end{equation}
 where $B$ is the Euclidean ball centered at $0$ with radius $1$,
and $|\cdot|$ is the max norm in $\R^3$, as in \eqref{eq.norminRs}.
 
 Obviously,
 this problem takes the form
 \eqref{nellbvp-intro} with
 (here and throughout,
 we denote the points of $\R^6$ by 
 $\bldw = (\mathbf{w}_1,\mathbf{w}_2)$, with $\bldw_1,\bldw_2\in \R^3$)
 \begin{itemize}
  \item[(i)] $\mathcal{O} := B$; \vspace*{0.05cm}
  \item[(ii)] $\LL_1 = \LL_2 = -\Delta$; \vspace*{0.05cm}
  \item[(iii)] $f_1:\overline{B}\times\R^2\times\R^6\to\R, \quad
  f_1(x,\bldz,\bldw) = e^{z_1}(1+|\bldw_2|^2)$; \vspace*{0.05cm}
  \item[(iv)] $f_2: \overline{B}\times\R^2\times\R^6\to\R, \quad
  f_2(x,\bldz,\bldw) = (16-z_2^2)\cos(\langle \bldw_1, \bldw_2\rangle)$; \vspace*{0.05cm}
  \item[(v)] $h_1:C^1(\overline{B},\R^2)\to \R, 
  \quad h_1[u_1,u_2] := u_1(0)+u_2(0)$; \vspace*{0.05cm}
  \item[(vi)] $h_2:C^1(\overline{B},\R^2)\to \R, \quad h_1[u_1,u_2] := 
  \int_{\de B}u_1^2(1-|\nabla u_2|^2)\,\d\sigma$;
  \item[(vii)] $\zeta_1\equiv \zeta_2 \equiv 1$.
 \end{itemize}
 Furthermore, it is straightforward to check that all the
 structural assumptions (I)-to-(VI) listed at the beginning of 
 Section \ref{sec.existenceenon} are satisfied (for every
 $\alpha\in(0,1)$). We now aim to show that, in this case,
 also assumptions (a)-to-(c) in statement of Theorem \ref{thmsol}
 are fulfilled. \medskip
 
 \textbf{Assumption (a).} To begin with, we consider the
 finite sequence 
 \begin{equation} \label{eq.choicerhoexmexistence}
  \varrho = \{\rho_1,\rho_2\}, \qquad\text{where $\rho_1 = \rho_2 = \sqrt{\frac{\pi}{6}}$}.
 \end{equation}
 Clearly, the function $f_1$ is continuous and non-negative
 on $\overline{B}\times I(\varrho)\times R(\varrho)$ (see \eqref{eq.defiIBvarrho} for the definition
 of $I(\varrho)$ and $R(\varrho)$); moreover,
 since $\rho_1,\rho_2\leq 4$ and since, by Cauchy-Schwarz inequality, we have
 (remind the definition of $|\cdot|$ in \eqref{eq.norminRs})
 \begin{align*}
  |\langle \bldw_1,\bldw_2\rangle|\leq 3\,|\bldw_1|\cdot|\bldw_2|\leq \frac{\pi}{2}
  \qquad\text{for any $\bldw = (\bldw_1,\bldw_2)\in R(\varrho)$},
 \end{align*}
 we easily deduce that also $f_2$ is (continuous and) non-negative
 on $\overline{B}\times I(\varrho)\times R(\varrho)$.
 
 As for the operators $h_1,h_2$, it is immediate to check that
 they are (continuous and) non-negative when restricted to the cone
 $ P(\varrho)$ (note that, if $\bldu\in P(\varrho)$, we have
 $|\nabla u_2| \leq \rho_2 < 1$); furthermore, 
 since $\bldu = (u_1,u_2)\in P(\varrho)$ implies
 that $0\leq u_1,u_2\leq\sqrt{\pi/6}$, we have
 \begin{align} \label{eq.estimh1h2exmexistence}
  h_1[\bldu] & = h_1[u_1,u_2] \leq 2\sqrt{\frac{\pi}{6}} \qquad\text{and}\qquad
  h_2[\bldu] = h_2[u_1,u_2]\leq \frac{\pi}{6}\big|\de B\big| = \frac{2\pi^2}{3}.
 \end{align}
 Thus, $h_1,h_2$ are bounded on $P(\varrho)$, and this proves
 that assumption (a) is fulfilled. \medskip
 
 \textbf{Assumption (b).} First of all we observe that, by definition, one has
 $$f_1(x,\bldz,\bldw) \geq e^{z_1} \qquad\text{for every
 $(x,\bldz,\bldw)\in\overline{B}\times\R^2\times\R^6$};$$
 as a consequence, given any $\delta > 0$,
 it is possible to find
 a small $\rho_0 = \rho_0(\delta) \in (0,\sqrt{\pi/6})$ such that
 (here, $I_0 = [0,\rho_0]\times [0,\rho_0]$ and 
 $R_0:=R_{\rho_0}\times R_{\rho_0}$, see \eqref{eq.defiIBvarrho})
 $$f_1(x,\bldz,\bldw) \geq e^{z_1}\geq \delta z_1\qquad\text{for every $(x,\bldz,\bldw)\in \overline{B}
 \times I_0\times R_0$}.$$
 This proves that $f_1$ satisfies
 \eqref{eq.mainestimk0}, and thus assumption (b) is fulfilled (with
 $k_0 = 1$). \medskip
 
 \textbf{Assumption (c).} We begin by explicitly computing
 the quantities appearing in \eqref{eq.defiMkHk}. On the one hand,
 by the very definition of $f_1,f_2$ we have
 \begin{equation} \label{eq.M12exmexistence}
  M_1 = \max_{\overline{B}\times I(\varrho)\times R(\varrho)}f_1
 = e^{\sqrt{\pi/6}}\Big(1+\frac{\pi}{6}\Big)\qquad\text{and}\qquad
 M_2 = \max_{\overline{B}\times I(\varrho)\times R(\varrho)}f_2 = 16.
 \end{equation}
 On the other hand, on account of \eqref{eq.estimh1h2exmexistence}, we have
 (notice that the constant function defined on $\overline{B}$ by
 $\bldu := (\sqrt{\pi/6},0)$
 certainly belongs to $P(\varrho)$)
 \begin{equation} \label{eq.H12exmexistence}
  H_1 = \sup_{\bldu\in P(\varrho)}h_1[\bldu]
 = 2\sqrt{\frac{\pi}{6}}\qquad\text{and}\qquad 
 H_2 = \sup_{\bldu\in P(\varrho)}h_2[\bldu] = \frac{2\pi^2}{3}.
 \end{equation}
 We now observe that, since $\LL_1 = \LL_2 = -\Delta$
 (and taking into account the very
 definition of Green operator, see \eqref{eq.defoperatorG}), one obviously has
 $$\GG_{1}(\hat{1}) = \GG_{\LL_1}(\hat{1}) = \GG_{(-\Delta)}(\hat{1})
 \qquad\text{and} \qquad
 \GG_{2}(\hat{1}) = \GG_{\LL_2}(\hat{1}) = \GG_{(-\Delta)}(\hat{1}),$$
 where $\GG_{(-\Delta)}(\hat{1})$ is the unique solution
 of
 $$\begin{cases}
 -\Delta u = 1 & \text{in $B$}, \\
 u\big|_{\de B} = 0.
 \end{cases}
 $$
 As a consequence, since a direct computation gives 
 $\GG_{(-\Delta)}(\hat{1}) = \frac{1}{2}(1-\|x\|^2)$, we get
 \begin{equation} \label{eq.Green12exmexistence}
 \|\GG_1(\hat{1})\|_{\infty} = \|\GG_2(\hat{1})\|_\infty = \frac{1}{2}.
 \end{equation}
 Analogously, since $\zeta_1\equiv\zeta_2\equiv 1$
 (and again since $\LL_1 = \LL_2 = -\Delta$),
 from \eqref{eqellipticbc} we deduce that
 $\gamma_1 = \gamma_2 = \hat{\gamma}$,
 where $\hat{\gamma}$ is the unique solution of
 $$
 \begin{cases}
   \Delta u = 0& \text{in $B$}, \\
 u\big|_{\de B} = 1.
 \end{cases} 
 $$
 As a consequence, since $\hat{\gamma} \equiv 1$ clearly solves
 the above problem, we get
 \begin{equation} \label{eq.gamma12exmexistence}
  \|\gamma_1\|_\infty = \|\gamma_2\|_\infty = 1.
  \end{equation}
 Finally, according to \eqref{intgreen}, we turn to provide an explicit estimate for
 $$\sup_{x \in B} 
 \int_{B}\left|\de_{x_l}g_{(-\Delta)}(y;x)\right|\,\d y
 \qquad(\text{with $l = 1,2,3$}),$$
 where $g_{(-\Delta)}$ is the Green function for $(-\Delta)$
 (and related to $B$). To this end,
 we make crucial use of the \emph{explicit}
 expression of $g_{(-\Delta)}$ (see, e.g., \cite[Section 2.2.4-(c)]{Evans})):
 \begin{equation} \label{eq.explicitgDelta}
  g_{(-\Delta)}(y;x) = 
 \frac{1}{4\pi}\,\bigg(
 \|x-y\|^{-1}-\Big(1 + \|x\|^2\,\|y\|^2-2\langle x,y\rangle\Big)^{-1/2}\bigg)
 \end{equation}
 where $\|\cdot\|$ is the usual Euclidean norm in $\R^3$. 
 Starting from \eqref{eq.explicitgDelta}, a direct yet tedious computation
 shows that (for every $x,y\in B$ with $x\neq y$)
 $$|\de_{x_l}g_{(-\Delta)}(y;x)| \leq \frac{1}{2\pi\|x-y\|^2};$$
 as a consequence, for every $x\in B$ we have
 \begin{align*}
  & \int_{B}\left|\de_{x_l}g_{(-\Delta)}(y;x)\right|\,\d y 
  \leq \frac{1}{2\pi}\int_B\|x-y\|^{-2}\,\d y
  \leq \frac{1}{2\pi}\int_{\{\|x-y\|<2\}}\|x-y\|^{-2}\,\d y \\
  & \qquad = \frac{1}{2\pi}\,\int_{\{\|y\|<2\}}\|y\|^{-2}\,\d y 
  = \frac{1}{2\pi}\,\big|\de B\big|\,\int_0^{2}\d\rho  = 4.
 \end{align*}
 Thus, taking into account that $\LL_1 = \LL_2 = -\Delta$, we obtain
 \begin{equation} \label{eq.G12exmexistence}
 \GG_{1,l} = \GG_{2,l}
 = \sup_{x \in B} \int_{B}\left|\de_{x_l}g_{(-\Delta)}(y;x)\right|\,\d y
 \leq 4, \qquad\text{for every $l = 1,2,3$}.
 \end{equation}
 By gathering together
 \eqref{eq.choicerhoexmexistence}, 
 \eqref{eq.M12exmexistence}, \eqref{eq.H12exmexistence},
 \eqref{eq.Green12exmexistence}, \eqref{eq.gamma12exmexistence}
 and \eqref{eq.G12exmexistence}, we are finally entitled
 to apply Theorem \ref{thmsol}: for any $\lambda_1 > 0$
 and any $\lambda_2,\eta_1,\eta_2 \geq 0$
 satisfying
 \begin{align*}
 & (\ast)\qquad\frac{\lambda_1}{2}\,e^{\sqrt{\pi/6}}\Big(1+\frac{\pi}{6}\Big)
  + 2\,\eta_1\sqrt{\frac{\pi}{6}} \leq \sqrt{\frac{\pi}{6}} \qquad\quad
  (\text{see assumption $\mathrm{(c)}_2$}), \\[0.05cm]
  & (\ast)\qquad \lambda_2 + \frac{2\pi^2}{3}\,\eta_2 \leq \sqrt{\frac{\pi}{6}}; 
  \qquad\quad (\text{see assumption $\mathrm{(c)}_2$}), \\[0.05cm]
  & (\ast)\qquad
  \max\bigg\{4\lambda_1\,e^{\sqrt{\pi/6}}\Big(1+\frac{\pi}{6}\Big),
  64\lambda_2\bigg\} \leq \sqrt{\frac{\pi}{6}}, \qquad\quad
  (\text{see assumption $\mathrm{(c)}_3$}), \\[0.05cm]
 \end{align*}
   there exists at least one solution $\bldu = (u_1,u_2)\in C^1(\overline{B},\R^2)$
  of \eqref{eq.BVPexistence} such that
  $$\|u_1\|_\infty,\,\|u_2\|_\infty\leq \sqrt{\frac{\pi}{6}} \qquad
  \text{and} \qquad \|\bldu\|_{C^1(\overline{B},\R^2)}\geq \rho_0.$$
  Here, $\rho_0  = \rho_0(\delta) > 0$ is as in assumption (b) and
  $\delta > 0$ is such that $\mu_1\leq\delta\lambda_1$
  (see assumption $\mathrm{(c)_1}$ and remind that $\mu_1 > 0$
 denotes the inverse of the spectral radius of $\LL_1 = -\Delta$,
 see \eqref{eq.defivarphik}). 
  It should be noticed that, since
  \eqref{eq.mainestimk0} holds \emph{for any given
  $\delta > 0$} (by accordingly choosing $\rho_0 = \rho_0(\delta) > 0$), 
  there is no need to have an explicit knowledge of $\mu_1$.
\end{ex}
\begin{ex} \label{exm.nonexistence}
 On Euclidean space $\R^3$, let us consider the following BVP
 \begin{equation} \label{eq.BVPnonexistence}
  \begin{cases}
   -\Delta u_1 = \lambda_1\,u_1^2\big(1-e^{-|\nabla u_2|}\big) & \text{in $B$}, \\
   -\Delta u_2 = \lambda_2\,\sin(u_2)\big(u_1^3+|\langle \nabla u_1,\nabla u_2\rangle|\big)
   & \text{in $B$}, \\
   u_1\big|_{\de B} = \eta_1\int_Bu_2^2\,\d x, \\
   u_2\big|_{\de B} = \eta_2\max\limits_{\de B}u_1,
  \end{cases}
 \end{equation}
 where $B$ is the Euclidean ball with centre $0$ and radius $1$ and 
 we adopt the same notation of Example \ref{exm.existence}.
 
 Obviously,
 this problem takes the form
 \eqref{nellbvp-intro} with
 \begin{itemize}
  \item[(i)] $\mathcal{O} := B$; \vspace*{0.05cm}
  \item[(ii)] $\LL_1 = \LL_2 = -\Delta$; \vspace*{0.05cm}
  \item[(iii)] $f_1:\overline{B}\times\R^2\times\R^6\to\R, \quad
  f_1(x,\bldz,\bldw) = z_1^2(1-e^{|\bldw_2|})$; \vspace*{0.05cm}
  \item[(iv)] $f_2: \overline{B}\times\R^2\times\R^6\to\R, \quad
  f_2(x,\bldz,\bldw) = \sin(z_2)(z_1^3+|\langle \bldw_1,\bldw_2\rangle|)$; \vspace*{0.05cm}
  \item[(v)] $h_1:C^1(\overline{B},\R^2)\to \R, 
  \quad h_1[u_1,u_2] := \int_B u_2^2\,\d x$; \vspace*{0.05cm}
  \item[(vi)] $h_2:C^1(\overline{B},\R^2)\to \R, \quad h_1[u_1,u_2] := \max\limits_{\de B}u_1$;
  \item[(vii)] $\zeta_1\equiv \zeta_2 \equiv 1$.
 \end{itemize}
 Furthermore, it is straightforward to check that all the
 structural assumptions (I)-to-(VI) listed at the beginning of 
 Section \ref{sec.existenceenon} are satisfied (for every
 $\alpha\in(0,1)$). We now aim to show that, in this case,
 assumptions (a)-to-(c) in statement of Theorem \ref{thmnonex}
 are fulfilled. \medskip
 
 \textbf{Assumption (a).} To begin with, we consider the
 finite sequence 
 \begin{equation} \label{eq.choicerhononexmexistence}
  \varrho = \{\rho_1,\rho_2\}, \qquad\text{where $\rho_1 = \rho_2 = 1$}.
 \end{equation}
 Clearly, the function $f_1$ is continuous and non-negative
 on $\overline{B}\times I(\varrho)\times R(\varrho)$; moreover, 
 for every  $(x,\bldz,\bldw)\in
   \overline{B}\times I(\varrho)\times R(\varrho)$
 one has (notice that, if $z\in I(\varrho)$, then $0\leq z_1\leq 1$) 
 \begin{equation} \label{eq.f1assumptionanonexistence}
   0\leq f_1(x,\bldz,\bldw) = z_1\cdot \big(z_1(1-e^{-|\bldw_2|})\big)
   \leq u_1.
 \end{equation}  
 Thus, $f_1$ fulfills assumption (a) (with $\tau_1 = 1$).
 
 As regards $f_2$, we obviously have that
 also this function is continuous and non-negative on 
 $\overline{B}\times I(\varrho)\times R(\varrho)$; moreover, 
 since $0\leq \sin(t)\leq t$ for every $0\leq t\leq 1$, we have
 \begin{equation} \label{eq.f2assumptionanonexistence}
  \begin{split}
   & 0\leq f_2(x,\bldz,\bldw) \leq z_2\big(1+|\langle 
   \bldw_1, \bldw_2\rangle|\big) \\
   & \qquad (\text{by Cauchy-Schwarz inequality, see Example \ref{exm.existence}}) \\
   & \qquad \leq z_2\big(1+3\,|\bldw_1|\cdot|\bldw_2|\big) \\
   & \qquad (\text{since $w = (\bldw_1,\bldw_2)\in R(\varrho)$ implies that 
   $|\bldw_1|,|\bldw_2|\leq 1$}) \\
   & \qquad \leq 4z_2 \qquad\qquad
   (\text{for every  $(x,z,w)\in
   \overline{B}\times I(\varrho)\times R(\varrho)$}).
  \end{split}
 \end{equation}
 As a consequence, also $f_2$ satisfies assumption (a)
 (with $\tau_2 = 4$). \medskip
 
 \textbf{Assumption (b).} First of all, it is very
 easy to check that both $h_1$ and $h_2$ are continuous and non-negative
 when restricted to the cone $P(\varrho)\subseteq C^1(\overline{B},\R)$;
 moreover, since the condition $\bldu = (u_1,u_2)\in P(\varrho)$ implies that 
 $0\leq u_1,u_2\leq 1$, we get
 \begin{equation} \label{h1assumptionbnonexistence}
  h_1[\bldu] = h_1[u_1,u_2] \leq \int_Bu_2\,\d x \leq
  \big(\max_{\overline{B}}u_2\big)\cdot|B_1| \leq \frac{4\pi}{3}\,\|\bldu\|_\infty,
 \end{equation}
 and this proves that $h_1$ fulfills assumption (b) (with $\xi_1 = (4\pi)/3$).
 
 Finally, by exploiting the very definition of $\|\cdot\|_\infty$, we have
 \begin{equation} \label{h2assumptionbnonexistence}
  h_2[\bldu] = h_2[u_1,u_2] = \max_{\de B}u_1 \leq \|\bldu\|_\infty,
 \end{equation}
 and thus also $h_2$ satisfies assumption (b) (with $\xi_2 = 1$). \medskip
 
 \textbf{Assumption (c).} By making use of all the computations already carried
 out in the previous Example \ref{exm.existence}, we know that
 (see, precisely, \eqref{eq.Green12exmexistence} and
 \eqref{eq.gamma12exmexistence}) \medskip
 
 (i)\,\,$\|\GG_1(\hat{1})\|_\infty = \|\GG_2(\hat{1})\|_\infty
 = 1/2$; \medskip
 
 (ii)\,\,$\|\gamma_1\|_\infty = \|\gamma_2\|_\infty = 1$. \medskip
 
 \noindent As a consequence, by gathering together
 \eqref{eq.choicerhononexmexistence}, 
 \eqref{eq.f1assumptionanonexistence}, \eqref{eq.f2assumptionanonexistence},
 \eqref{h1assumptionbnonexistence}, \eqref{h2assumptionbnonexistence}
 and the above (i)-(ii), we are entitled to apply
 Theorem \ref{thmnonex}: for any $\lambda_1,\lambda_2,\eta_1,\eta_2\geq 0$ satisfying
  \begin{align*}
  \frac{\lambda_1}{2}+\frac{4\pi}{3}\,\eta_1 < 1 \qquad\text{and} \qquad
  2\lambda_2 + \eta_2 < 1,
   \end{align*}
   the BVP \eqref{eq.BVPnonexistence} possesses only the zero solution
   (notice that $\bldu \equiv 0$ trivially solves
   \eqref{eq.BVPnonexistence}).
\end{ex}



\begin{thebibliography}{xxx}

\bibitem{Alves-Moussaoui}
C.O.  Alves and A.  Moussaoui, Existence of solutions for a class of singular elliptic systems with convection term, 
\textit{Asymptot. Anal.}, \textbf{90} (2014),  237--248. 

\bibitem{Amann-rev} H. Amann, 
Fixed point equations and nonlinear eigenvalue
problems in ordered Banach spaces, \textit{SIAM. Rev.}, \textbf{18} (1976),
620--709.

\bibitem{AmCra} H. Amann and M.G. Crandall, 
On some existence theorems for semi-linear elliptic equations,
\textit{Indiana Univ. Math. J.}, \textbf{27} (1978), no. 5, 779--790.

\bibitem{BT} 
H. Br\'ezis and R.E.L. Turner, On a class of superlinear elliptic problems, 
\textit{Comm. Partial Differential Equations}, \textbf{2} (1977) 601--614.

\bibitem{genupa}
F. Cianciaruso, G. Infante and P. Pietramala, Solutions of perturbed Hammerstein integral equations with applications, \textit{Nonlinear Anal. Real World Appl.}, \textbf{33} (2017), 317--347.

\bibitem{genupa2}
F. Cianciaruso, G. Infante and P. Pietramala, Non-zero radial solutions for elliptic systems with coupled functional BCs in exterior domains, \textit{Proc. Edinb. Math. Soc.}, \textbf{62} (2019), 747--769.

\bibitem{Radu-book}
D. O'Regan and R. Precup, 
\textit{Theorems of Leray-Schauder type and applications},
Series in Mathematical Analysis and Applications, 3. Gordon and Breach Science Publishers, Amsterdam, 2001.

\bibitem{Fig18}
D.G. de Figueiredo, \textit{Nonvariational Semilinear Elliptic Systems}. In: Lavor C., Gomes F. (eds) Advances in 
Mathematics and Applications. Springer, Cham, (2018), 131--151.

\bibitem{FGM}
D.G. De Figueiredo, M. Girardi and M. Matzeu, Semilinear elliptic equations with dependence on the gradient via mountain-pass techniques, \textit{Differential Integral Equations}, \textbf{17} (2004), 119--126. 

\bibitem{FY}
D.G. De Figueiredo and J. Yang, A priori bounds for positive solutions of a non-variational elliptic system, \textit{Comm. Partial Differential Equations}, \textbf{26} (2001), 2305--2321. 

\bibitem{Evans}
L.C. Evans,
\textit{Partial differential equations},
American Mathematical Society, Providence, RI, (2010).

\bibitem{GT} 
D. Gilbarg and N.S. Trudinger,
\textit{Elliptic Partial Differential Equations of Second Order},
Springer-Verlag, Berlin, (2011).

\bibitem{GM} 
M. Girardi and M. Matzeu, Positive and negative solutions of a quasi-linear elliptic equation by a mountain pass method
and truncature techniques, \textit{Nonlinear Anal.}, \textbf{59} (2004), 199--210.

\bibitem{Goodrich3}
C. S. Goodrich, 
New Harnack inequalities and existence theorems for radially symmetric solutions of elliptic PDEs with sign changing or vanishing Green's function, \textit{J. Differential Equations}, \textbf{264} (2018), 236--262.

\bibitem{Goodrich4}
C. S. Goodrich, 
Radially symmetric solutions of elliptic PDEs with uniformly negative weight, \textit{Ann. Mat. Pura Appl.}, \textbf{197} (2018), 1585--1611.

\bibitem{guolak} D. Guo and V. Lakshmikantham,
\textit{Nonlinear problems in abstract cones}, Academic Press, Boston,
(1988). 

\bibitem{gi-tmna} G. Infante, Nonzero positive solutions of a multi-parameter elliptic system with functional BCs, \textit{Topol. Methods Nonlinear Anal.}, \textbf{52} (2018), 665--675.

\bibitem{gi-jepe} G. Infante, Nonzero positive solutions of nonlocal elliptic systems with functional BCs, \textit{J. Elliptic Parabol. Equ.}, \textbf{5} (2019), 493--505.

\bibitem{gi-nieto} G. Infante, Positive and increasing solutions of perturbed Hammerstein integral equations with derivative dependence, 
\textit{Discrete Contin. Dyn. Syst. Ser. B}, \textbf{25} (2020), 691--699.

\bibitem{gi-ho}
G. Infante, Positive solutions of systems of perturbed Hammerstein integral equations with arbitrary order dependence, \textit{Philos. Trans. Roy. Soc. A}, to appear.

\bibitem{KimSak}
S. Kim and G. Sakellaris,
\textit{Green's function for second order elliptic equations with singular lower order coefficients},
\textit{Comm. Partial Differential Equations}, \textbf{44}
(2019), 228--270.

\bibitem{MawSchm}
J. Mawhin and K. Schmitt, Upper and lower solutions and semilinear second order elliptic equations with nonlinear boundary conditions, \textit{Proc. Roy. Soc. Edinburgh Sect. A}, \textbf{97} (1984), 199--207. 

\bibitem{MawSchm-c}
J. Mawhin and K. Schmitt, 
Corrigendum to ``Upper and lower solutions and semilinear second order elliptic equations with nonlinear boundary conditions'', \textit{Proc. Roy. Soc. Edinburgh Sect. A}, \textbf{100} (1985), 361. 

\bibitem{Pao-Wang}
C. V. Pao and Y. M. Wang,
Nonlinear fourth-order elliptic equations with nonlocal boundary conditions, \textit{J. 
Math. Anal. Appl.}, \textbf{372} (2010), 351--365.

\bibitem{Po} S. I. Pokhozhaev, 
On equations of the form $-\Delta u=f(x,u,Du)$,
\textit{Math. USSR, Sb.}, \textbf{41} (1982), 269--280. 

\bibitem{RuSu}
D. Ruiz and A. Su\'arez, Existence and uniqueness of positive solution of a logistic equation with nonlinear gradient term, \textit{Proc. Roy. Soc. Edinburgh Sect. A}, \textbf{137} (2007), 555--566. 

\bibitem{WaDe}
X.J. Wang and Y.B. Deng, 
Existence of multiple solutions to nonlinear elliptic equations of nondivergence form,
\textit{J. Math. Anal. Appl.}, \textbf{189} (1995), 617--630. 

\bibitem{Webb}
J. R. L. Webb, Existence of positive solutions for a thermostat model, \textit{Nonlinear Anal. Real World
Appl.}, \textbf{13} (2012), 923--938.

\bibitem{Yan}
Z.Q. Yan, A note on the solvability in $W^{2,p}(\Omega)$ for the equation $-\Delta u=f(x,u,Du)$,
\textit{Nonlinear Anal.}, \textbf{24} (1995), 1413--1416.

\end{thebibliography}
\end{document}